\documentclass[11pt,reqno]{amsart}
\usepackage{graphicx} 
\usepackage{amsfonts}
\usepackage{mathtools}
\usepackage{amssymb}
\usepackage{amsthm}
\usepackage{mathrsfs}
\usepackage[dvipsnames]{xcolor}
\usepackage[pdfencoding=auto]{hyperref}

\usepackage[a4paper,top=2.5cm,bottom=2.5cm,left=2.5cm,right=2.5cm]{geometry}
\usepackage{hyphenat}
\usepackage{tikz}
\usepackage{pgfplots}
\pgfplotsset{compat=1.18}
\usepackage{xpatch}
\usepackage{afterpage}
\usepackage{faktor}
\usepackage{mathabx}
\usepackage{enumitem}
\usepackage{parskip}
\setlist[enumerate]{label=\textit{\arabic*}.}

\theoremstyle{plain}
\newtheorem{theorem}{Theorem}
\newtheorem{lemma}[theorem]{Lemma}
\newtheorem{proposition}[theorem]{Proposition}
\newtheorem{corollary}[theorem]{Corollary}
\theoremstyle{definition}
\newtheorem*{remark}{Remark}

\numberwithin{equation}{section}

\renewcommand{\d}{\text{\rm d}}

\renewcommand{\Im}{\operatorname{Im}}
\renewcommand{\Re}{\operatorname{Re}}
\newcommand{\Fq}{\mathbb{F}_q}

\title{Maximizers of the $L^2\to L^4$ Fourier extension inequality for cones in finite fields}

\author[C. González-Riquelme]{Cristian González-Riquelme}
\address{Centre de Recerca Matemàtica, Campus de Bellaterra, Edifici C 
08193 Bellaterra (Barcelona), Spain.}
\email{cgonzalez@crm.cat}

\author[T. Ismoilov]{Tolibjon Ismoilov}
\address{SISSA - Scuola Internazionale Superiore di Studi Avanzati\\ Via Bonomea 265, 34136 Trieste, Italy}
\email{tolibjon.ismoilov@sissa.it}


\begin{document}

\begin{abstract}
Sharp Fourier restriction theory and finite field extension theory have both been topics of interest in the last decades. Very recently, in \cite{GonzalezOliveira}, the research into the intersection of these two topics started. There it was established that, for the $(3,1)$-cone $\Gamma_{(3,1)}^3:=\{\boldsymbol{\eta}\in \Fq^4\setminus\{\boldsymbol{0}\} : \eta_1^2+\eta_2^2+\eta_3^2=\eta_4^2\},$ the Fourier extension map from $L^2\to L^{4}$ is maximized by constant functions when $q=3\, \pmod{4}$. 
In this manuscript, we advance this line of inquiry by establishing sharp inequalities for the $L^{2}\to L^{4}$ extension inequalities applicable for all remaining cones $\Gamma^3\subset \mathbb{F}_q^4$. These cones include the $(2,2)$-cone $\Gamma_{(2,2)}^3:=\{\boldsymbol{\eta}\in \Fq^4\setminus\{\boldsymbol{0}\} : \eta_1^2+\eta_2^2=\eta_3^2+\eta_4^2\}$
for general $q=p^n$ and the $(3,1)$-cone
when $q=1\, \pmod{4}$. Moreover, we classify all the extremizers in each case. We note that the analogous problem for the $(2, 2)$-cone in the
euclidean setting remains open.
\end{abstract}
\allowdisplaybreaks
\maketitle

\section{Introduction}
\subsection{Background} 
Fourier Restriction theory has been a major topic of research since the work of Stein (see \cite{stein1993harmonic}), where the connection between curvature and decay of the Fourier transform was first introduced. Later, in the work of Strichartz \cite{Strichartz1977}, Fourier restriction estimates for quadratic surfaces, which also include cones, were studied. Let $d\ge 3$, given a pair of positive integers $a,b$ with $a+b=d$, we define the $2$-sheeted $(a,b)$-cone as \begin{align*}
\Gamma^{d-1}_{(a,b)}:=\left\{(x_1, \dots, x_d)\in \mathbb{R}^{d}\, :\;  \sum_{i=1}^{a}x_i^2=\sum_{j=a+1}^{d}x_j^2\right\}.
\end{align*} 
We call the $1$-sheeted cone the subset of this cone which consists of all points with the last coordinate positive.  
Defining the measure $$\d \mu(x_1,\dots,x_d)=\frac{\d x_1\dots \d x_{d-1}}{|x_d|},$$
over $\Gamma^{d-1}_{(a,b)}$, Strichartz proved that there exists a constant $C_{\Gamma^{d-1}_{(a,b)}}$ such that for any function $f$ defined on the $(a,b)$-cone and  
$$\mathcal{E}f(x)\coloneqq\int_{\Gamma^{d-1}_{(a,b)}}f(\omega)e^{-ix\cdot \omega}\d \mu(\omega),$$
satisfies \begin{align}\label{strichartzinrd}
\|\mathcal{E}f\|_{2+\frac{4}{d}}\le C_{\Gamma^{d-1}_{(a,b)}}\|f\|_2.    
\end{align}
These inequalities are connected to the homogeneous wave equation (see \cite{foschioliveira}). 
Sharp restriction theory has been developed since the work of Foschi \cite{foschicone}, where the optimal constant for the inequality \eqref{strichartzinrd} is achieved for the $1$-sheeted $(2,1)$-cone and the $2$-sheeted $(3,1)$-cone (see also \cite{carneiroimrn}). In either case, it is crucial that the Lebesgue exponent $2+\frac{4}{d}$ is an even exponent, since in such instances Plancherel's theorem allows us to translate the problem into a problem of convolutions. In such a context, considerable effort has been made to understand: related sharp inequalities, existence and stability of extremizers (see \cite{ChristShaoExistenceOfExtremizers, FoschiSphereRestrictionMaximizers, goncalvesgiuseppe, giuseppe, thielenegrooliveira} and the references therein).

\subsection{Fourier Restriction in finite fields}
 Mockenhaupt and Tao \cite{MockenhauptTao} initiated the study of restriction phenomenon in the space $\mathbb{F}_{q}^d$, where $\mathbb{F}_{q}$ is a finite field with characteristic $\operatorname{char}(\mathbb{F})>2$. Given $1\le r,s\le \infty$ and $\mathcal{S}\subset \mathbb{F}_q^d$, let us define $\textbf{R}^{*}_{\mathcal{S}}(r\to s)$ as the smallest constant such that the inequality 
 \begin{align}\label{extensioninequalityintro}
\|(f\sigma)^{\vee}\|_{L^{s}(\mathbb{F}^d_q,\d \boldsymbol{x})}\le \textbf{R}^{*}_{\mathcal{S}}(r\to s)\|f\|_{L^{r}(\mathcal{S},\d \sigma)}     
 \end{align}
holds for every $f: \mathcal{S}\to \mathbb{C}$; here $\d \boldsymbol{x}$ is the usual counting measure in $\mathbb{F}^{d}_q$, and $\d \sigma$ is the normalized counting measure in $\mathcal{S}$ and 
\begin{equation*}
    (f\sigma)^{\vee}(\boldsymbol{x})=\frac{1}{|\mathcal{S}|}\underset{\boldsymbol{\xi}\in \mathcal{S}}{\sum}f(\boldsymbol{\xi})e(\boldsymbol{\xi}\cdot\boldsymbol{x}),
\end{equation*}
where $e:\mathbb{F}_q\to \mathbb{S}^{1}$ is a non-principal character. A non-principal character is a non-constant map $e:\mathbb{F}_q\to \mathbb{S}^{1}$ such that $e(x+y)=e(x)e(y),$ these can be listed as:
$e_{a}(\cdot)=\exp \left(\frac{2\pi i}{p}\operatorname{Tr}_n(a\, \cdot)\right),$ where $a\in \mathbb{F}_q$ and $\operatorname{Tr}_n:\mathbb{F}_q\to \mathbb{F}_p$ is the $\mathbb{F}_p$-linear map given by
\begin{align*}
\operatorname{Tr}_n(x)=\sum_{k=0}^{n-1}x^{p^k}.  
\end{align*}

We say that we have \textit{restriction property} in $\mathcal{S}$ with exponents $r,s$ if the constant $\textbf{R}^{*}_{\mathcal{S}}(r\to s)$ is bounded independently of $q$. The restriction problem over finite fields asks, for a given $\mathcal{S}$, for which $r,s$ such property holds. In \cite{MockenhauptTao} the authors established the \textit{restriction property}  $2\to 4$ for the paraboloid in low dimensions ($\mathcal{P}^1$ and $\mathcal{P}^2$). Let us define in $\mathbb{F}^{d}_q$ the cone with signature $(a,b)$ (with $a,b\in \mathbb{Z}_{>0}$, $a+b=d$) as:
\begin{align*}
\Gamma^{d-1}_{(a,b)}:=\left\{(\eta_1,\dots,\eta_d)\in \mathbb{F}_{q}^{d}\, : \; \sum_{i=1}^{a}\eta_i^2=\sum_{j=a+1}^{d}\eta_j^2\right\}.
\end{align*}
For the $(2,1)$ and $(3,1)$ cones, the $2\to 4$ \textit{restriction property} was studied in \cite{MockenhauptTao} and \cite{kohleepham}, respectively. 

Very recently, the study of sharp constants for Fourier extension inequalities was inaugurated in the work of the first author along with Oliveira e Silva \cite{GonzalezOliveira}. Here, the authors compute the optimal constant $\textbf{R}^{*}_{\mathcal{S}}(2\to 4)$ for $\mathcal{S}=\mathcal{P}^{2}$, for general $q$, and $\mathcal{S}=\Gamma^{3}_{(3,1)}\setminus \{\boldsymbol{0}\}$ when $q=3\, \pmod{4}$. Moreover, they compute the optimal constant $\textbf{R}^{*}_{\mathcal{P}^1}(2\to 6)$. In all listed cases, constants are maximizers of the inequality. Furthermore, in all listed cases, the maximizers are required to be constant in absolute value. The complex phase of the maximizers is a more intricate issue that was settled just for $\mathcal{P}^2$ when $q=1\pmod{4}$. It is important to notice that the removal of the origin in the cone is of major relevance (that is not the case in \cite{MockenhauptTao}), since constants are not maximizers if it is not removed (see \cite[Theorem 1.6]{GonzalezOliveira}). This highlights the delicate nature of these optimal inequalities: even small perturbations on the objects 
of study could change the whole behavior of the problem. Notably, the case $q=1 \pmod{4}$ for $\textbf{R}^{*}_{\Gamma^{3}_{(3,1)}\setminus \{\boldsymbol{0}\}}(2\to 4)$ was beyond the reach of the methods there presented. 

Another direction of research in sharp extension inequalities of finite fields has been developed in the work \cite{BiswasCarneiroFlockOliveiraeSilvaStovallTautges}, where optimal inequalities for the Fourier extension inequalities for the moment curve in finite fields are achieved. 

\subsection{Main results}

In the present manuscript, our main results are the following.
\begin{theorem}[Optimal constant]\label{thm:sharprestriction}
We have that  
\begin{align*}
\textbf{\emph{R}}^\ast_{\Gamma^{3}_{(2,2)}\setminus \{\boldsymbol{0}\}}(2\to 4)=q\left(\frac{q^5+4q^4-4q^3-6q^2+3q+3}{(q+1)^6(q-1)^3}\right)^{\frac{1}{4}}.\end{align*}
\end{theorem}

\begin{theorem}[Classification of the extremizers]\label{thm:classification}
    The function $f:\Gamma_{(2,2)}^3\setminus \{\boldsymbol{0}\}\to \mathbb{C}$ is a maximizer of \eqref{extensioninequalityintro} if and only if $f$ is given by 
    \begin{equation}\label{formulacharacterization}
        f(\eta_1, \eta_2, \eta_3, \eta_4)=\lambda\cdot \exp{\left(\frac{2 \pi i}{p}\operatorname{Tr}_n(a_1\eta_1+a_2\eta_2+a_3\eta_3+a_4\eta_4)\right)},
    \end{equation}
    where $\lambda\in \mathbb{C}^\times$ and $a_i\in \mathbb{F}_q$.
\end{theorem}
We emphasize that these results are not available in the euclidean setup; in contrast, the $\Gamma^{3}_{(3,1)}$ cone (in $\mathbb{R}^4$) was established by Foschi \cite{foschicone} (as mentioned before). This evidences the big distinction that can exist between these two geometric objects. We aim, by our results in this discrete setting, to shed light on the problem in the euclidean setting.  

Moreover, our results imply the remaining case which has not been established before for the $(3,1)$-cone. Noticing that $\Gamma^{3}_{(2,2)}=\Gamma^{2}_{(3,1)}$ when $q=1 \pmod{4}$ (given that in this case there exists $\omega\in \mathbb{F}_q$ such that $\omega^2=-1$), we have the following. 
\begin{corollary}
We have that, when $q=1\pmod{4}$,  $$\textbf{\emph{R}}^{*}_{\Gamma^{3}_{(3,1)}\setminus \{\boldsymbol{0}\}}(2\to 4)=q\left(\frac{q^5+4q^4-4q^3-6q^2+3q+3}{(q+1)^6(q-1)^3}\right)^{\frac{1}{4}},$$
where this optimal constant is achieved by $f:\Gamma_{(2,2)}^{3}\setminus \{\boldsymbol{0}\}\to \mathbb{C}$ if and only if $f$ is given by 
    \begin{equation*}
        f(\eta_1, \eta_2, \eta_3, \eta_4)=\lambda\cdot \exp{\left(\frac{2 \pi i}{p}\operatorname{Tr}_n(a_1\eta_1+a_2\eta_2+a_3\eta_3+a_4\eta_4)\right)},
    \end{equation*}
    where $\lambda\in \mathbb{C}^\times$ and $a_i\in \mathbb{F}_q$.
\end{corollary}

These results complete the analysis of sharp constants for the analogues of the Strichartz's estimates in $\mathbb{F}_q^4$. Henceforth, we write $$\Gamma_{(2,2)}^{3}\setminus \{\boldsymbol{0}\}=:\Gamma^{3}.$$  
By using \cite[Proposition 2.1]{GonzalezOliveira}, the claims of Theorem \ref{thm:sharprestriction} and Theorem \ref{thm:classification} are equivalent to establishing that the smallest constant $\textbf{C}_{\Gamma^{3}}(2\to 4)$ such that, $\forall f:\Gamma^3\to \mathbb{C}$
\begin{equation}\label{eq:combinatorialestimate}
   \sum_{\boldsymbol{\xi}\in \mathbb{F}_q^4}\bigg\lvert \sum_{\substack{\boldsymbol{\eta}_1, \boldsymbol{\eta}_2\, \in \, \Gamma^3 \\ \boldsymbol{\eta}_1+\boldsymbol{\eta}_2=\boldsymbol{\xi}}} f(\boldsymbol{\eta}_1)f(\boldsymbol{\eta}_2) \bigg\rvert^2\leq \mathbf{C}_{\Gamma^3}(2\to4) \bigg(\sum_{\boldsymbol{\xi}\in \Gamma^3} \lvert f(\boldsymbol{\xi})\rvert^2\bigg)^2 
\end{equation}
holds, is \begin{align*}
    \mathbf{C}_{\Gamma^3}(2\to4)=\frac{q^5+4q^4-4q^3-6q^2+3q+3}{(q+1)^2(q-1)},
    \end{align*}
and that the extremizers for this inequality are characterized by \eqref{formulacharacterization}.

In \cite{GonzalezOliveira} the method used to prove the analogous of \eqref{eq:combinatorialestimate} for the $(3,1)$ cone when $q=3 \pmod 4$ depends heavily on the fact that in that case, given $\boldsymbol{\xi}\in \Gamma^{3}_{(3,1)}\setminus \{\boldsymbol{0}\},$ the set of pairs $(\boldsymbol{\xi}_1,\boldsymbol{\xi}_2)\in \left(\Gamma^3_{(3,1)}\setminus \{\boldsymbol{0}\}\right)^2$ such that $\boldsymbol{\xi}_1+\boldsymbol{\xi}_2=\boldsymbol{\xi}$ are pairs of points in the same line $\{\lambda\boldsymbol{\xi}\; :\; \lambda \in \mathbb{F}^{\times}_q\}$. That rigid and simple structure, along with the fact that these lines provide a disjoint partition of the cone, turns out to be very useful for the estimates required there. In contrast, this property does not hold when $q=1 \pmod 4$ or the cone has signature $(2,2).$ Here, the structure of those pairs is more intricate and does not provide a partition of the set. Therefore, a more refined analysis is required.  

Our approach consists of, first, describing the cone in a way more amenable to our purposes. Then, determine precisely, for a given $\boldsymbol{\xi}\in \mathbb{F}_q^4$, the cardinality of the sets $$\Sigma_{\boldsymbol{\xi}}:=\left\{(\boldsymbol{\xi}_1,\boldsymbol{\xi}_2)\in \left(\Gamma^{3}\right)^{2}\, : \; \boldsymbol{\xi}_1+\boldsymbol{\xi}_2=\boldsymbol{\xi}\right\}.$$
After that, for a given $\boldsymbol{\xi}\in \Gamma^3$, to completely characterize the structures of $\Sigma_{\boldsymbol{\xi}}$. These sets can be understood as the union of two punctured planes that share a punctured line. The union of all $\underset{\boldsymbol{\xi}\in \Gamma^{3}}{\bigcup}\Sigma_{\boldsymbol{\xi}}$ describes a foliation of the cone. Following this, we observe that the value of $|\Sigma_{\boldsymbol{\xi}}|$ depends only on whether $\boldsymbol{\xi}$ is in $\Gamma^{3}$, $\{\boldsymbol{0}\}$ or $\mathbb{F}_q^{4}\setminus \left(\Gamma^3\cup \{\boldsymbol{0}\}\right).$ We call this third option the {\it generic points}. All this is described in Section \ref{sec: geometric structure}. 

After that, in Section \ref{sec:proof of thm1}, we use the symmetries of the problem in order to reduce the question to the even and non-negative functions $f$.  Then,  we estimate the contribution of a generic $\boldsymbol{\xi}$ on the LHS of \eqref{eq:combinatorialestimate} by using Cauchy--Schwarz on the inner term of the sum. Then we use the appropriate algebraic manipulation in order to concentrate the problem into estimates of points in the set $\Sigma_{\boldsymbol{\xi}}$ for $\boldsymbol{\xi}$ on the cone (and $\boldsymbol{0}$). To estimate the contribution of these remaining points is what concentrates the efforts in the present manuscript. For this, the already mentioned structure of $\Sigma_{\boldsymbol{\xi}}$ for points in the cone plays a crucial role.
 
In Section \ref{sec:classification}, we provide the classification of the maximizers stated in Theorem \ref{thm:classification}. For this, we use the characterization of each of the inequalities used in Section \ref{sec:proof of thm1}, and it allows us to prove that the maximizers are constant in absolute value and that their complex phase must fulfill a functional equation: the product of the phases of a pair of points in the cone only depends on the sum of these points. This rigid structure enables us to demonstrate that, over the planes contained in the cone, these phases behave almost like characters. To conclude our results, we study the interaction of disjoint punctured planes contained in the cone.  
 
\section*{Acknowledgements}
The first author was supported by the Spanish Ministry of Science, Innovation and Universities,  grant PID2023-150984NB-I00. CGR is also supported by the Spanish State Research Agency, through the Severo Ochoa and María de Maeztu Program for Centers and Units of Excellence in R\&D (CEX2020-001084-M) and thank CERCA Programme/Generalitat de Catalunya for institutional support. The authors are grateful to Emanuel Carneiro and Diogo Oliveira e Silva for helpful discussions. The second author is thankful to Ajay Gautam and Awais Gujjar for their remarks regarding the Segre embedding. 

\section*{Notation} For a finite set $A$ we denote its cardinality by $|A|$. Real and imaginary parts of a given complex number $z\in\mathbb{C}$ are  denoted by $\Re(z)$ and $\Im(z)$, and 
the principal value of the argument is $\operatorname{Arg}(z)\in(-\pi,\pi]$.
If $\mathcal{F}$ is a finite set of variables, then $C(\mathcal{F})$ denotes a quantity that only depends on elements of $\mathcal{F}$. For any $\boldsymbol{\rho}\in \mathbb{F}_{q}^{d}$, we use the notation $\boldsymbol{\rho}\coloneqq(\rho_1,\rho_2, \dots, \rho_d)$

Throughout this manuscript, we assume $p$ is an odd prime and $q=p^n$ for some positive integer $n$.

\section{Geometric structure of \texorpdfstring{$\Sigma_{\boldsymbol{\xi}}$}{}}\label{sec: geometric structure}
First let us observe that, by the change of variables $(\eta_1,\eta_2,\eta_3,\eta_4)\mapsto (\eta_1-\eta_3,\eta_1+\eta_3,\eta_4-\eta_2,\eta_4+\eta_2),$ we can use the definition $$\Gamma^3:=\left\{(\eta_1,\eta_2,\eta_3,\eta_4)\in \mathbb{F}^{4}_q\setminus \{\boldsymbol{0}\}\,:\; \eta_1\eta_2=\eta_3\eta_4\right\}.$$ 
 The first thing we need to observe is that there are exactly $(q-1)(q+1)^2$ points in $\Gamma^3$. This follows immediately by considering two cases: $\eta_1=0$ and $\eta_1\neq 0$, and then counting all the possible values for the remaining coordinates.

As mentioned in the introduction, one of the key points of our argument is to provide a clear description of a certain collection of points that appear in the estimate. It turns out there is a natural parametrization of the cone $\Gamma^3$ which gives exactly that.
\begin{lemma} There is a bijection between $\Gamma^3$ and  the  Cartesian product $\mathbb{F}_q^\times \! \times \!\mathbb{P}^1\!\times\!\mathbb{P}^1$, where 
\begin{align*}
    \mathbb{P}^1&\coloneqq\left\{(1, y)\in \mathbb{F}^2_q\; |\; y\in \mathbb{F}_q\right\}\cup\left\{(0, 1)\right\}
\end{align*}
denotes one-dimensional projective space over $\mathbb{F}_q$.
\end{lemma}

\begin{proof}
Consider the map $\pi:\mathbb{F}_q^\times \! \times \!\mathbb{P}^1\!\times\!\mathbb{P}^1 \to \Gamma^3$ given by 
    \begin{equation}\label{eq:parametrization}
   \pi(\lambda, \alpha, \beta)\coloneqq (\lambda\alpha_1\beta_1, \lambda \alpha_2\beta_2, \lambda \alpha_1\beta_2, \lambda\alpha_2\beta_1),
    \end{equation}
for every $(\lambda, \alpha, \beta)\in \mathbb{F}_q^\times \! \times \!\mathbb{P}^1\!\times\!\mathbb{P}^1$, where $\alpha=(\alpha_1, \alpha_2)$ and $\beta=(\beta_1, \beta_2)$.
We claim that $\pi$ is a bijection. Note the following identities, 
\begin{align*}
    &\pi(\lambda, (1, \alpha_2), (1,\beta_2))=(\lambda, \lambda\alpha_2\beta_2, \lambda\beta_2, \lambda\alpha_2);\\
    &\pi(\lambda, (1, \alpha_2), (0,1))=(0, \lambda\alpha_2, \lambda, 0);\\
    &\pi(\lambda, (0, 1), (1,\beta_2))= (0, \lambda\beta_2, 0, \lambda);\\
    &\pi(\lambda, (0, 1), (0, 1))=(0, \lambda, 0, 0).
\end{align*}
Using these identities we conclude that if $\pi(\lambda, \alpha, \beta)=\pi(\mu, \gamma, \delta)$ for some $(\lambda, \alpha, \beta)$ and $(\mu, \gamma, \delta)$ in $\mathbb{F}_q^\times \! \times \!\mathbb{P}^1\!\times\!\mathbb{P}^1$, then $\alpha=\gamma$, $\beta=\delta$ and $\lambda=\mu$. Consequently, $\pi$ is an injective function. In the other direction, since the order of $\Gamma^3$ is $(q-1)(q+1)^2$, which is the same as $\lvert \mathbb{F}_q^\times \! \times \!\mathbb{P}^1\!\times\!\mathbb{P}^1\rvert$, it can be concluded that $\pi$ is bijective.\qedhere
\end{proof}

\begin{remark}
   This parametrization is a special case of a general family of maps known as Segre embeddings in projective spaces \cite{Segre1891}.
\end{remark}

Now, let us proceed with the study of the sets $\Sigma_{\boldsymbol{\xi}}.$ 
\begin{proposition}\label{prop:cardinality of E_x}
For each $\boldsymbol{\xi}\in \Fq^4$ the cardinality of the set $\Sigma_{\boldsymbol{\xi}}$ can be expressed as follows
    \begin{equation*}
        \lvert \Sigma_{\boldsymbol{\xi}}\rvert=\left\{\begin{aligned}
            &(q+1)^2(q-1),& & \text{ if } \boldsymbol{\xi}=0,\\
            &2q^2-q-2,& & \text{ if }  \boldsymbol{\xi} \in \Gamma^3,\\
            &q^2+q, & & \text{ otherwise} .
        \end{aligned}\right.
    \end{equation*}
\end{proposition}

\begin{proof}
First of all, when $\boldsymbol{\xi}=0$ we see that $\Sigma_{\boldsymbol{0}}=\{(\boldsymbol{\rho}, -\boldsymbol{\rho}): \boldsymbol{\rho}\in \Gamma^3 \}$, which clearly has the same order as $\Gamma^3$, i.e. $\lvert E_{\boldsymbol{0}}\rvert=(q-1)(q+1)^2$. Now, suppose $\boldsymbol{\xi}\neq \boldsymbol{0}$. 
It is natural to consider the sets $H_{\boldsymbol{\xi}}$ given by 
\begin{equation}\label{eq:defintion of H_xi}
    H_{\boldsymbol{\xi}}\coloneqq\left\{\boldsymbol{ \eta }\in \Gamma^3\; \colon \; (\xi_1-\eta_1)(\xi_2-\eta_2)=(\xi_3-\eta_3)(\xi_4-\eta_4)\right\}.
\end{equation}
If $\boldsymbol{\xi} \in \Gamma^3$ then we see that $\boldsymbol{\xi}\in H_{\boldsymbol{\xi}}$. Since $(\boldsymbol{\xi}, \boldsymbol{0})$ is not an element of $\Sigma_{\boldsymbol{\xi}}$, it should be omitted in the count of $\Sigma_{\boldsymbol{\xi}}$. Since that is the only point we need to omit, we have the relation $\lvert \Sigma_{\boldsymbol{\xi}}\rvert=\lvert H_{\boldsymbol{\xi}}\rvert-1$. This is one of the points where the parametrization \eqref{eq:parametrization} has a clear advantage over the Cartesian coordinates. Suppose $\boldsymbol{\xi}, \boldsymbol{\eta} \in\Gamma^3,$ then  $\boldsymbol{\xi}=\pi(\lambda, \alpha,\beta)$ and $\boldsymbol{\eta}=\pi(\mu, \gamma,\delta)$. Hence, the condition in \eqref{eq:defintion of H_xi} is equivalent to the following
\begin{align*}
 \alpha_1\beta_1\gamma_2\delta_2+\alpha_2\beta_2\gamma_1\delta_1= \alpha_1\beta_2\gamma_2\delta_1+\alpha_2\beta_1\gamma_1\delta_2.
\end{align*}
By factoring out this equation, we see that it is satisfied if and only if at least one of the following identities holds: $\alpha=\gamma$ or $\beta=\delta$. Since $\lvert \mathbb{P}^1\rvert=q+1$ we conclude that $\lvert H_{\boldsymbol{\xi}}\rvert=(q-1)(2q+1)$ for each $\boldsymbol{\xi}\in \Gamma^3$. Thus, $\lvert \Sigma_{\boldsymbol{\xi}}\rvert=2q^2-q-2$ if $\boldsymbol{\xi}\in \Gamma^3$.

Now, suppose $\boldsymbol{\xi}\notin\Gamma^3$ and $\boldsymbol{\xi}\neq \boldsymbol{0}$. Notice that $\boldsymbol{\xi}\not\in H_{\boldsymbol{\xi}}$ in this case, therefore $\lvert \Sigma_{\boldsymbol{\xi}}\rvert=\lvert H_{\boldsymbol{\xi}}\rvert$. It is not restrictive to assume $\xi_1=1$ due to the symmetries of the equation in \eqref{eq:defintion of H_xi}. Hence, we can isolate $\eta_2$ in the equation and deduce
\begin{equation*}
    \eta_2=\xi_4\eta_3+\xi_3\eta_4-\xi_2\eta_1+\xi_2-\xi_3\xi_4.
\end{equation*}
We need to plug this into the equation $\eta_1\eta_2=\eta_3\eta_4$ to ensure that the solution $\eta$ belongs to the cone. The resulting equation is the following
\begin{equation*}
    (\xi_3\eta_1-\eta_3)\eta_4=\eta_1(\xi_2\eta_1-\xi_4\eta_3+\xi_3\xi_4-\xi_2)
\end{equation*}
If $\eta_3=\xi_3\eta_1$ then, the equation simplifies to $\eta_1(\eta_1-1)(\xi_2-\xi_3\xi_4)=0$. Since $\boldsymbol{\xi}\not \in \Gamma^3$ we see that $\eta_1=0$ or $\eta_1=1$ and for each case $\eta_4$ can be arbitrary number in $\Fq$. Hence, we get $2q$ different solutions in this situation. If $\eta_3\neq \xi_3\eta_1$ then, for each  $q$ different values of $\eta_1$, the coordinate $\eta_3$ can take $q-1$ different values. Hence, in this situation we have $q^2-q$ many solutions. Combining both situations we deduce that the set $H_{\boldsymbol{\xi}}$ has exactly $q^2+q$ elements.
\end{proof}

\begin{remark}
    Note that the set $H_{\boldsymbol{\xi}}$ has a structure easier to deal with when $\boldsymbol{\xi} \in\Gamma^3$ compared to the other situation. In such a case, the following properties of the set $H_{\boldsymbol{\xi}}$ must be noted.
    \begin{enumerate}[label=\textnormal{(}\it{\roman*}\textnormal{)}]
        \item If $\boldsymbol{\eta} \in H_{\boldsymbol{\xi}}$ then $\lambda \boldsymbol{\eta}\in H_{\boldsymbol{\xi}}$, for each $\lambda\in \Fq^\times$; Also, $H_{\boldsymbol{\xi}}=H_{\lambda\boldsymbol{\xi}},$ for any $\lambda\in \mathbb{F}_q^\times$.
        \item $\boldsymbol{\eta} \in H_{\boldsymbol{\xi}}\setminus\{\boldsymbol{\xi}\} \iff  \boldsymbol{\xi} -\boldsymbol{\eta} \in H_{\boldsymbol{\xi}}\setminus\{\boldsymbol{\xi}\};$
        \item $\boldsymbol{\eta}\in H_{\boldsymbol{\xi}} \iff \boldsymbol{\xi} \in H_{\boldsymbol{\eta}}$
    \end{enumerate}
\end{remark}

We are going to look closely at the sets $H_{\boldsymbol{\xi}}$. 
Recall that for each $\boldsymbol{\xi}\in \Gamma^3$ we established in the proof of Proposition \ref{prop:cardinality of E_x} that $H_{\boldsymbol{\xi}}$ is made of two parts.
If $\boldsymbol{\xi}=\pi(\lambda, \alpha,\beta)$ then we have seen that 
\begin{equation*}
    H_{\boldsymbol{\xi}}=\{\pi(\mu, \alpha, y) : \; (\mu, y)\in \mathbb{F}_q^\times\!\times\! \mathbb{P}^1\}\cup \{\pi(\mu, x, \beta) : \; (\mu, x)\in \mathbb{F}_q^\times\!\times\! \mathbb{P}^1\}.
\end{equation*}
Since these sets have geometric interpretations and are very crucial for our estimates, we further denote the following disjoint subsets:
\begin{align*}
    H_{\boldsymbol{\xi}}^+&\coloneqq\{\pi(\mu, \alpha, y) : \; \mu\in \mathbb{F}_q^\times, \;\; y\in \mathbb{P}^1\setminus\{\beta\}\},\\
    H_{\boldsymbol{\xi}}^-&\coloneqq\{\pi(\mu, x, \beta) : \; \mu\in \mathbb{F}_q^\times, \;\; x\in \mathbb{P}^1\setminus\{\alpha\}\},\\
\shortintertext{ and the line }
    \mathcal{L}_{\boldsymbol{\xi}}&\coloneqq\{\mu \boldsymbol{\xi} : \; \mu \in \mathbb{F}_q^\times\}=\{\pi(\mu, \alpha, \beta) : \; \mu\in \mathbb{F}_q^\times\}.
\end{align*}
It is now clear that $H_{\boldsymbol{\xi}}= H_{\boldsymbol{\xi}}^+\cup H_{\boldsymbol{\xi}}^-\cup \mathcal{L}_{\boldsymbol{\xi}}$ and $\lvert  H_{\boldsymbol{\xi}}^\pm\rvert=q(q-1)$ and $\lvert \mathcal{L}_{\boldsymbol{\xi}}\rvert=q-1$.

Look at the sets $H_{\boldsymbol{\xi}}^+\cup \mathcal{L}_{\boldsymbol{\xi}}$ and $H_{\boldsymbol{\xi}}^-\cup \mathcal{L}_{\boldsymbol{\xi}}$ in Cartesian coordinates for the case $\xi_i\neq 0, i=1, 2, 3, 4$. 
\begin{align*}
        H_{\boldsymbol{\xi}}^+\cup \mathcal{L}_{\boldsymbol{\xi}}&=\{\pi(\mu, \alpha, y) : \; (\mu, y)\in \mathbb{F}_q^\times\!\times\! \mathbb{P}^1\}\\
        &=\{(t\xi_1, s\xi_2, s\xi_3, t\xi_4): t, s\in \Fq \text{ and } (t, s)\neq (0,0)\},
    \shortintertext{and similarly}
         H_{\boldsymbol{\xi}}^-\cup \mathcal{L}_{\boldsymbol{\xi}}&=\{(t\xi_1, s\xi_2, t\xi_3, s\xi_4): t, s\in \Fq \text{ and } (t, s)\neq (0,0)\}.
\end{align*}
Note that this parametrization will be slightly different if some of the coordinates of $\boldsymbol{\xi}$ are zeros. However, one can see that in all cases, the sets $H_{\boldsymbol{\xi}}^+\cup \mathcal{L}_{\boldsymbol{\xi}}$ and $H_{\boldsymbol{\xi}}^-\cup \mathcal{L}_{\boldsymbol{\xi}}$ are almost closed under the addition of $\Fq^4$; we need to include $\boldsymbol{0}$ to achieve this property. Therefore, geometrically, they represent punctured planes in $\mathbb{F}_q^4$. For simplicity, we denote these sets by $A_{\boldsymbol{\xi}}^+$ and $A_{\boldsymbol{\xi}}^-$, respectively.

\begin{proposition}\label{prop: properties of planes}
    For points $\boldsymbol{\xi}$ and $\boldsymbol{\eta}$ in the cone $\Gamma^3$ the following statements hold
    \begin{enumerate}[label=\textnormal{(}\textit{\roman*}\textnormal{)}]
    \itemsep-0.3em 
         \item $H_{\boldsymbol{\xi}}^\pm=H_{\lambda \boldsymbol{\xi}}^\pm$ and $A_{\boldsymbol{\xi}}^\pm=A_{\lambda \boldsymbol{\xi}}^\pm$, for all $\lambda \in \Fq^\times$;
        \item $\boldsymbol{\eta}\in H_{\boldsymbol{\xi}}^\pm \iff \boldsymbol{\xi}\in H_{\boldsymbol{\eta}}^\pm$; 
        \item $\boldsymbol{\eta} \in H_{\boldsymbol{\xi}}^\pm \iff \boldsymbol{\xi} +\boldsymbol{\eta} \in H_{\boldsymbol{\xi}}^\pm$;
        \item The sets $A_{\boldsymbol{\xi}}^\pm$ and $A_{\boldsymbol{\eta}}^\pm$ are either disjoint or they are the same;
        \item The sets $A_{\boldsymbol{\xi}}^\pm$ and $A_{\boldsymbol{\eta}}^\mp$ always intersect along a punctured line, where the punctured point is $\boldsymbol{0}$ in $\mathbb{F}_q^4$.
    \end{enumerate}
\end{proposition}
\begin{proof}
Part $(i)$ follows immediately from the definitions of the sets $H_{\boldsymbol{\xi}}^\pm$. 
In the rest of the proof, we only focus on the cases $H_{\boldsymbol{\xi}}^+$ and $A_{\boldsymbol{\xi}}^+$; the other cases follow in an analogous way. Now, let $\boldsymbol{\xi}=\pi(\lambda, \alpha, \beta)$ and assume $\boldsymbol{\eta} \in H_{\boldsymbol{\xi}}^+$, this means $\boldsymbol{\eta}=\pi(\mu, \alpha, \gamma)$ for some $(\mu, \alpha, \gamma)\in \mathbb{F}_q^\times\!\times\!\mathbb{P}^1\!\times\!\mathbb{P}^1$ such that $\gamma\neq \beta$. Since $\boldsymbol{\xi}$ and $\boldsymbol{\eta}$ have a common second coordinate in this parametrization, it concludes the proof of $(ii)$. To prove $(iii)$, consider the sum 
\begin{equation*}
\begin{split}
    \boldsymbol{\xi}+\boldsymbol{\eta}&=\pi(\lambda, \alpha, \beta)+\pi(\mu, \alpha, \gamma)\\
    &=\bigg(\alpha_1(\lambda \beta_1+\mu \gamma_1), \alpha_2(\lambda \beta_2+\mu \gamma_2), \alpha_1(\lambda \beta_2+\mu \gamma_2), \alpha_2(\lambda \beta_1+\mu \gamma_1) \bigg)\\
    &=\left\{ \begin{aligned}
        &\pi\left(\lambda \beta_1+\mu \gamma_1, \alpha, \left(1, \frac{\lambda \beta_2+\mu \gamma_2}{\lambda \beta_1+\mu \gamma_1}\right)\right), &\text{ if } &\lambda \beta_1+\mu \gamma_1\neq 0\\
        & \pi\bigg(\lambda \beta_2+\mu \gamma_2, \alpha, \left(0, 1\right)\bigg), & \text{ if } &\lambda \beta_1+\mu \gamma_1=0
    \end{aligned}\right..
\end{split}
\end{equation*}
It shows that, in both cases, $\boldsymbol{\xi}+\boldsymbol{\eta}$ belongs to the set $H_{\boldsymbol{\xi}}^+$. Now suppose that $\boldsymbol{\xi}+\boldsymbol{\eta}\in H_{\boldsymbol{\xi}}^+$ for some $\boldsymbol{\eta}\in \Gamma^3$, then $-\boldsymbol{\xi}-\boldsymbol{\eta}$ also belongs to the same set. By what we have already proved, we obtain that  $-\boldsymbol{\eta}=\boldsymbol{\xi}+(-\boldsymbol{\xi}-\boldsymbol{\eta})\in H_{\boldsymbol{\xi}}^+ \implies \boldsymbol{\eta}\in H_{\boldsymbol{\xi}}^+$. 

For part $(iv)$, note that all the elements in $A_{\boldsymbol{\xi}}^+$ have the same second coordinate as $\boldsymbol{\xi}$ in the parametrization $\pi$. Therefore, if there is at least one element in the intersection $A_{\boldsymbol{\xi}}^+\cap A_{\boldsymbol{\eta}}^+$, we deduce that all of the elements in each of the sets have the same second coordinate, hence they coincide. 

Finally, for point $(v)$, assume that $ \boldsymbol{\xi} = \pi(\lambda, \alpha, \beta) $ and $ \boldsymbol{\eta} = \pi(\mu, \gamma, \delta) $. For a point $ \boldsymbol{\rho} \in \Gamma^3 $ to be an element of $ A_{\boldsymbol{\xi}}^+ \cap A_{\boldsymbol{\eta}}^- $, it is both necessary and sufficient that $ \boldsymbol{\rho} = \pi(\sigma, \alpha, \delta) $ for some $ \sigma \in \mathbb{F}_q^\times $. This collection of points is precisely a punctured line, as stated in the assertion.
\end{proof}
\begin{proposition}\label{prop:mixedproductidentity}
The following identity holds
    \begin{equation*}\label{eq:aplusproductaminusidentity}
    \sum_{\boldsymbol{\xi} \in \Gamma^3}\bigg(\sum\limits_{\boldsymbol{\eta}_1\in A^+_{\boldsymbol{\xi}}} f(\boldsymbol{\eta}_1)^2\bigg)\bigg(\sum\limits_{\boldsymbol{\eta}_2\in A^-_{\boldsymbol{\xi}}} f(\boldsymbol{\eta}_2)^2\bigg)= (q-1)\bigg(\sum_{\boldsymbol{\xi}\in\Gamma^3}  f(\boldsymbol{\xi}) ^2\bigg)^2.
    \end{equation*}
\end{proposition}
\begin{proof}
One can rewrite each of the sums on the left-hand side using the parametrization $\pi$. Consider the terms that correspond to $\boldsymbol{\xi}=\pi(\lambda, \alpha, \beta)$
\begin{align*}
        \sum\limits_{\boldsymbol{\eta}_1\in A^+_{\boldsymbol{\xi}}}f(\boldsymbol{\eta}_1)^2=\sum_{\substack{\mu\in \mathbb{F}_q^\times\\ y\in \mathbb{P}^1}} f(\pi(\mu, \alpha, y))^2\\
\shortintertext{ and }
  \sum\limits_{\boldsymbol{\eta}_1\in A^-_{\boldsymbol{\xi}}}f(\boldsymbol{\eta}_1)^2=\sum_{\substack{\mu\in \mathbb{F}_q^\times\\ x\in \mathbb{P}^1}} f(\pi(\mu, x, \beta))^2.
\end{align*}
Thus, the sum on the left looks like the following
\begin{align*}
     \sum_{\substack{\lambda\in \mathbb{F}_q^\times\\ \alpha, \beta \in \mathbb{P}^1}} & \bigg(\sum_{\substack{\mu_1\in \mathbb{F}_q^\times\\ y\in \mathbb{P}^1}} f(\pi(\mu_1, \alpha, y))^2\bigg) \bigg(\sum_{\substack{\mu_2\in \mathbb{F}_q^\times\\ x\in \mathbb{P}^1}} f(\pi(\mu_2, x, \beta))^2\bigg)\\
     &= \sum_{\substack{\lambda\in \mathbb{F}_q^\times\\ \alpha, \beta \in \mathbb{P}^1}} \sum_{\substack{\mu_1, \mu_2\in \mathbb{F}_q^\times\\ x, y\in \mathbb{P}^1}} f(\pi(\mu_1, \alpha, y))^2 f(\pi(\mu_2, x, \beta))^2\\
     &=\sum_{\lambda\in \mathbb{F}_q^\times} \sum_{\substack{\mu_1\in \mathbb{F}_q^\times\\ \alpha, y\in \mathbb{P}^1}} \sum_{\substack{\mu_2\in \mathbb{F}_q^\times\\ x,\beta\in \mathbb{P}^1}}f(\pi(\mu_1, \alpha, y))^2 f(\pi(\mu_2, x, \beta))^2.
\end{align*}
It is now clear that the two inner sums can be separated, and they are identical sums. Therefore, we can conclude the desired identity.
\end{proof}

\section{Proof of Theorem \ref{thm:sharprestriction}} \label{sec:proof of thm1}
In this section, we provide the argument that allows us to conclude the sharp constant in \eqref{eq:combinatorialestimate}.

\emph{Step 1: Symmetrization}.
 To begin, we utilize the method of antipodal symmetrization, which appeared in works \cite{ChristShaoExistenceOfExtremizers, FoschiSphereRestrictionMaximizers}.
Note that right-hand side of \eqref{eq:combinatorialestimate} does not change if we replace $f(\boldsymbol{\xi})$ by $f_{\sharp}(\boldsymbol{\xi})\coloneqq\sqrt{\frac{\lvert f(\boldsymbol{\xi})\rvert^2 +\lvert f(-\boldsymbol{\xi})\rvert^2}{2}}$, for each $\boldsymbol{\xi}\in \Gamma^3$. Moreover, one can show this change does not reduce the left-hand side of the inequality. Indeed, 
\begin{equation*}
    \begin{split}
          \sum_{\boldsymbol{\xi}\in \mathbb{F}_q^4}\bigg\lvert \sum_{\substack{\boldsymbol{\eta}_1, \boldsymbol{\eta}_2\in \Gamma^3 \\ \boldsymbol{\eta}_1+\boldsymbol{\eta}_2=\boldsymbol{\xi}}} &f(\boldsymbol{\eta}_1)f(\boldsymbol{\eta}_2) \bigg\rvert^2\\
          &=\int_{(\Gamma^3)^4} f(\boldsymbol{\eta}_1)f(\boldsymbol{\eta}_2)\overline{f(-\boldsymbol{\eta}_3)}\overline{f(-\boldsymbol{\eta}_4)}\underbrace{\mathbf{\delta}(\boldsymbol{\eta}_1+\boldsymbol{\eta}_2+\boldsymbol{\eta}_3+\boldsymbol{\eta}_4)\d \boldsymbol{\eta}_1\d\boldsymbol{\eta}_2\d\boldsymbol{\eta}_3\d\boldsymbol{\eta}_4}_{\d \Sigma}\\
          &=\int_{(\Gamma^3)^4} f(\boldsymbol{\eta}_1)\overline{f(-\boldsymbol{\eta}_2)}f(\boldsymbol{\eta}_3)\overline{f(-\boldsymbol{\eta}_4)}{\d \Sigma}= Q(f, f, f, f),
    \end{split}
\end{equation*}
where 
\begin{equation*}
        Q(f_1, f_2, f_3, f_4)\coloneqq  \int_{(\Gamma^3)^4}{f_1(\boldsymbol{\eta}_1)} \overline{f_2(-\boldsymbol{\eta}_2)}f_3(\boldsymbol{\eta}_3)\overline{f_4(-\boldsymbol{\eta}_4)}{\d \Sigma}.
\end{equation*}
Symmetry of the measure $\d \Sigma$ allows us to write 
\begin{equation*}
     Q(f, f, f, f)= Q(f^\ast, f^\ast, f, f) 
\end{equation*}
where $f^\ast(\boldsymbol{\xi})\coloneqq\overline{f(-\boldsymbol{\xi})}$.
Hence, we take the average of these two expressions
\begin{equation*}
    \begin{split}
         Q(f, f, f, f)&=\int_{(\Gamma^3)^4} \bigg(\frac{f(\boldsymbol{\eta}_1)\overline{f(-\boldsymbol{\eta}_2)}+\overline{f(-\boldsymbol{\eta}_1)} f(\boldsymbol{\eta}_2)}{2}\bigg)f(\boldsymbol{\eta}_3)\overline{f(-\boldsymbol{\eta}_4)}{\d \Sigma}\\
         &=\operatorname{Re}\int_{(\Gamma^3)^4} \bigg(\frac{f(\boldsymbol{\eta}_1)\overline{f(-\boldsymbol{\eta}_2)}+\overline{f(-\boldsymbol{\eta}_1)} f(\boldsymbol{\eta}_2)}{2}\bigg)f(\boldsymbol{\eta}_3)\overline{f(-\boldsymbol{\eta}_4)}{\d \Sigma}.
    \end{split}
\end{equation*}
Using the inequality $\operatorname{Re}(z)\leq \lvert z\rvert$ and then Cauchy-Schwarz we get
\begin{equation*}
    \begin{split}
          Q(f, f, f, f)  &\leq \int_{(\Gamma^3)^4} \bigg\lvert\frac{f(\boldsymbol{\eta}_1)\overline{f(-\boldsymbol{\eta}_2)}+\overline{f(-\boldsymbol{\eta}_1)} f(\boldsymbol{\eta}_2)}{2}\bigg\rvert \left\lvert f(\boldsymbol{\eta}_3)\overline{f(-\boldsymbol{\eta}_4)}\right\rvert{\d \Sigma}\\
         &\stackrel{C-S}{\leq} \int_{(\Gamma^3)^4} f_{\sharp}(\boldsymbol{\eta}_1)f_{\sharp}(\boldsymbol{\eta}_2)\lvert f(\boldsymbol{\eta}_3)\rvert \lvert f(-\boldsymbol{\eta}_4)\rvert{\d \Sigma}=Q(f_\sharp, f_\sharp, \lvert f\rvert, \lvert f\rvert).
    \end{split}
\end{equation*}
Note that the equality is attained if and only if the following conditions are satisfied
\begin{itemize}
    \item $\left(f(\boldsymbol{\eta}_1)\overline{f(-\boldsymbol{\eta}_2)}+\overline{f(-\boldsymbol{\eta}_1)} f(\boldsymbol{\eta}_2)\right)f(\boldsymbol{\eta}_3)\overline{f(-\boldsymbol{\eta}_4)}\geq0 $, for all $\boldsymbol{\eta}_i\in \Gamma^3$ such that $\boldsymbol{\eta}_1+\boldsymbol{\eta}_2+\boldsymbol{\eta}_3+\boldsymbol{\eta}_4=0$. 
    \item $f(\boldsymbol{\eta}_1)=C(\boldsymbol{\eta}_1, \boldsymbol{\eta}_2){f(-\boldsymbol{\eta}_2)}$ and $\overline{f(-\boldsymbol{\eta}_1)}=C(\boldsymbol{\eta}_1, \boldsymbol{\eta}_2)\overline{f(\boldsymbol{\eta}_2)}$, for some $C(\boldsymbol{\eta}_1, \boldsymbol{\eta}_2)\in \mathbb{C}$, for all $\boldsymbol{\eta}_1, \boldsymbol{\eta}_2\in \Gamma^3$;
\end{itemize}

Also notice that $Q(f, f, f, f)=Q(f, f, f^\ast, f^\ast)$, therefore, in the same way we can show 
\begin{equation}\label{eq:multilinear form inequality}
     Q(f, f, f, f) \leq Q(f_\sharp, f_\sharp, \lvert f\rvert, \lvert f\rvert)\leq Q(f_\sharp, f_\sharp, f_\sharp, f_\sharp).
\end{equation}
The equality in the last inequality also occurs if $f$ satisfies the previous conditions. 

Henceforth, we may assume that the function $f$ is even and takes only non-negative values.


For simplicity, we introduce the following notation for expressions that appear frequently throughout the proof. Let $g$ be a function on $\Gamma^3$ for any $A\subseteq\Gamma^3$ and $\boldsymbol{\xi}\in \Fq^4$ we set
\begin{align*}
    \sum_{A} g&\coloneqq\sum_{\boldsymbol{\eta} \in A} g(\boldsymbol{\eta}),\\
    \sum_{A, \; \boldsymbol{\xi}} g \!\cdot \! g&\coloneqq \sum_{\substack{\boldsymbol{\eta} \in A\\ (\boldsymbol{\xi}-\boldsymbol{\eta})\in\Gamma^3}} g(\boldsymbol{\eta})g(\boldsymbol{\xi}-\boldsymbol{\eta}).
\end{align*}

\emph{Step 2: Mass transport.}
We split the sum on the left side of \eqref{eq:combinatorialestimate} into three parts $\Gamma^3$ and $\{\boldsymbol{0}\}$, and the rest of the points in $\mathbb{F}_q^4$, 
\begin{equation*}
    \begin{split}
        \sum_{\boldsymbol{\xi}\in \Fq^4} \bigg(\sum\limits_{\Gamma^3, \; \boldsymbol{\xi}}f \!\cdot \! f\bigg)^2=\sum\limits_{\boldsymbol{\xi} \in \Gamma^3}\bigg( \sum\limits_{\Gamma^3, \; \boldsymbol{\xi}}f \!\cdot \! f\bigg)^2 &+\bigg( \sum\limits_{\boldsymbol{\xi}\in \Gamma^3}f(\boldsymbol{\xi}) f(-\boldsymbol{\xi} )\bigg)^2+\sum\limits_{\boldsymbol{\xi} \not\in \Gamma_0^3}\bigg( \sum\limits_{\Gamma^3, \; \boldsymbol{\xi}}f \!\cdot \! f\bigg)^2,  \\
    \end{split}
\end{equation*}
where $\Gamma_0^3=\Gamma^3 \cup \{\boldsymbol{0}\}$.

To begin with, we estimate the terms outside the set $\Gamma_0^3$ by employing the Cauchy-Schwarz inequality alongside the precise formula laid out in Proposition \ref{prop:cardinality of E_x}. Subsequently, we complete the expression by strategically adding and subtracting suitable terms as follows
\begin{equation*}
    \begin{split}
        \sum\limits_{\boldsymbol{\xi} \not\in \Gamma_0^3}\bigg( \sum\limits_{\Gamma^3, \; \boldsymbol{\xi}}f \cdot f \bigg)^2&\leq \sum\limits_{\boldsymbol{\xi} \not\in \Gamma_0^3}\bigg(\sum\limits_{\Gamma^3, \; \boldsymbol{\xi}} 1 \! \cdot \! 1 \bigg)\bigg(\sum\limits_{\Gamma^3, \; \boldsymbol{\xi}}f^2 \! \cdot \! f^2 \bigg) = (q^2+ q) \sum_{\boldsymbol{\xi} \not\in \Gamma^3_0}  \bigg(\sum\limits_{\Gamma^3, \; \boldsymbol{\xi}}f^2 \! \cdot\!  f^2\bigg)\\
        &=(q^2+q)\sum\limits_{\boldsymbol{\xi} \in \mathbb{F}_q^4}\bigg(\sum\limits_{\Gamma^3, \; \boldsymbol{\xi}}f^2 \! \cdot \! f^2 \bigg) -(q^2+q)\sum\limits_{\boldsymbol{\xi} \in \Gamma_0^3}\bigg(\sum\limits_{\Gamma^3, \; \boldsymbol{\xi}}f^2 \! \cdot \! f^2 \bigg) \\
        &=(q^2+q)\bigg(\sum\limits_{\Gamma^3} f^2\bigg)^2-(q^2+q)\sum\limits_{\boldsymbol{\xi} \in \Gamma^3}\bigg(\sum\limits_{\Gamma^3, \; \boldsymbol{\xi}}f^2 \! \cdot \! f^2 \bigg) -(q^2+q)\sum\limits_{\Gamma^3}f^4,\\
    \end{split}
\end{equation*}
where the first term in the last line comes from the change of the order of summation in the corresponding sum above.
Observe that equality is achieved when, for every fixed $\boldsymbol{\xi} \not\in \Gamma_0^3$, the identity $f(\boldsymbol{\xi}-\boldsymbol{\eta})=C(\boldsymbol{\xi})f(\boldsymbol{\eta})$ holds for every $\boldsymbol{\eta} \in \Gamma^3$ with $\boldsymbol{\xi} - \boldsymbol{\eta} \in \Gamma^3$, where $C(\boldsymbol{\xi})$ is a constant depending solely on $\boldsymbol{\xi}$.

Based on this estimate, we deduce that to show \eqref{eq:combinatorialestimate} it suffices to prove the following inequality
\begin{equation}\label{eq:mainestimate}
    \begin{split}
        \sum\limits_{\boldsymbol{\xi} \in \Gamma^3}\bigg( \sum\limits_{H_{\boldsymbol{\xi}}, \;\boldsymbol{\xi}}&f \! \cdot \! f\bigg)^2 -q(q+1)\sum\limits_{\boldsymbol{\xi} \in \Gamma^3}\sum\limits_{H_{\boldsymbol{\xi}}, \;\boldsymbol{\xi}} f^2 \! \cdot \! f^2\\\
        &- q(q+1) \sum\limits_{\Gamma^3} f^4
        -\frac{2q^4-5q^3-5q^2+5q+4}{(q-1)(q+1)^2}\left(\sum\limits_{\Gamma^3} f^2\right)^2\leq 0.
\end{split}
\end{equation}

Now we utilize the geometric information about the sets $H_{\boldsymbol{\xi}}$. Rewrite the first term by separating the following sets in the inner sum $\mathcal{L}^\circ_{\boldsymbol{\xi}}\coloneqq \mathcal{L}_{\boldsymbol{\xi}}\setminus\{\boldsymbol{\xi}\}$, $H_{\boldsymbol{\xi}}^+$, and $H_{\boldsymbol{\xi}}^-$.
\begin{equation*}
    \begin{split}
        \sum\limits_{\boldsymbol{\xi} \in \Gamma^3}\bigg( \sum\limits_{H_{\boldsymbol{\xi}}, \;\boldsymbol{\xi}}f \! \cdot \! f\bigg)^2= &\sum\limits_{\boldsymbol{\xi} \in \Gamma^3}\bigg\{\bigg( \sum\limits_{ \mathcal{L}^\circ_{\boldsymbol{\xi}}, \; \boldsymbol{\xi}}f \! \cdot \! f\bigg)^2 +\bigg( \sum\limits_{H_{\boldsymbol{\xi}}^+, \; \boldsymbol{\xi}}f \! \cdot \! f\bigg)^2+\bigg( \sum\limits_{H_{\boldsymbol{\xi}}^-, \; \boldsymbol{\xi}}f \! \cdot \! f\bigg)^2 \bigg\}\\
        &+2\sum_{\boldsymbol{\xi}\in\Gamma^3}\bigg\{\bigg(\sum\limits_{\mathcal{L}^\circ_{\boldsymbol{\xi}}, \; \boldsymbol{\xi}}f \! \cdot \! f\bigg)\bigg(\sum\limits_{H_{\boldsymbol{\xi}}^+\cup H_{\boldsymbol{\xi}}^-, \; \boldsymbol{\xi}}f \! \cdot \! f\bigg)+\bigg(\sum\limits_{H_{\boldsymbol{\xi}}^+, \; \boldsymbol{\xi}}f \! \cdot \! f\bigg)\bigg(\sum\limits_{H_{\boldsymbol{\xi}}^-, \; \boldsymbol{\xi}}f \! \cdot \! f\bigg)\bigg\}.
    \end{split}
\end{equation*}
Similarly
\begin{equation*}
    \begin{split}
        \sum\limits_{\boldsymbol{\xi} \in \Gamma^3}\sum\limits_{H_{\boldsymbol{\xi}}, \; \boldsymbol{\xi}} f^2\!\cdot \! f^2=\sum\limits_{\boldsymbol{\xi} \in \Gamma^3} \sum\limits_{\mathcal{L}^\circ_{\boldsymbol{\xi}}, \; \boldsymbol{\xi}} f^2\!\cdot \! f^2 +\sum\limits_{\boldsymbol{\xi} \in \Gamma^3}\sum\limits_{H_{\boldsymbol{\xi}}^+\cup H_{\boldsymbol{\xi}}^-, \; \boldsymbol{\xi}}f^2\!\cdot \! f^2.
    \end{split}
\end{equation*}

\textit{Step 3: Application of Proposition  \ref{prop:mixedproductidentity}.} 
This is an important part of our argument. We can apply the Proposition \ref{prop:mixedproductidentity} for the last term on the right-hand side of the inequality \eqref{eq:mainestimate}. For simplicity, we denote by $M_q=\frac{2q^4-5q^3-5q^2+5q+4}{(q-1)(q+1)^2}$ and the last term of \eqref{eq:mainestimate} can be replaced by
\begin{align*}
       \frac{M_q}{q-1} \sum_{\boldsymbol{\xi}\in \Gamma^3}&\bigg(\sum_{A_{\boldsymbol{\xi}}^+} f^2\bigg)\bigg(\sum_{A_{\boldsymbol{\xi}}^-} f^2\bigg)\\
\shortintertext{which we can write as}
\frac{M_q}{q-1} \sum_{\boldsymbol{\xi}\in \Gamma^3} \bigg\{\bigg(\sum_{\mathcal{L}_{\boldsymbol{\xi}}}f^2\bigg)^2+\bigg(\sum_{\mathcal{L}_{\boldsymbol{\xi}}}f^2\bigg)&\bigg(\sum_{H_{\boldsymbol{\xi}}^+\cup H_{{\boldsymbol{\xi}}^-}} f^2\bigg)+\bigg(\sum_{H_{\boldsymbol{\xi}}^+} f^2\bigg)\bigg(\sum_{H_{\boldsymbol{\xi}}^-} f^2\bigg)\bigg\}.
\end{align*}

\textit{Step 4: Separation of the sums over $H_{\boldsymbol{\xi}}^+$ and $H_{\boldsymbol{\xi}}^-$.} 
We have rewritten the expression on the left of \eqref{eq:mainestimate}.
In this rewritten version, there are only two expressions where sums over $H_{\boldsymbol{\xi}}^+$ and $H_{\boldsymbol{\xi}}^-$ are multiplied by each other. We want to replace such terms with terms where there are no mixed products, using the AM-GM inequality.
Grouping them, we have 
\begin{equation*}
    \begin{split}
    2\sum_{\boldsymbol{\xi}\in\Gamma^3}\bigg(\sum\limits_{H_{\boldsymbol{\xi}}^+, \; \boldsymbol{\xi}}f  \!\cdot \! f\bigg)\bigg(\sum\limits_{H_{\boldsymbol{\xi}}^-, \; \xi}&f\!\cdot \! f\bigg)-\frac{M_q}{q-1}\bigg(\sum_{H_{\boldsymbol{\xi}}^+} f^2\bigg)\bigg(\sum_{H_{\boldsymbol{\xi}}^-} f^2\bigg) \\
    & \stackrel{C-S}{\leq}\bigg(2-\frac{M_q}{q-1}\bigg)\sum_{\boldsymbol{\xi}\in\Gamma^3}\bigg(\sum\limits_{H_{\boldsymbol{\xi}}^+, \; \boldsymbol{\xi}}f\!\cdot \! f\bigg)\bigg(\sum\limits_{H_{\boldsymbol{\xi}}^-, \; \xi}f\!\cdot \! f\bigg)\\
    &\stackrel{AM-GM}{\leq}\bigg(1-\frac{M_q}{2(q-1)}\bigg)\sum_{\boldsymbol{\xi}\in\Gamma^3}\bigg\{\bigg(\sum\limits_{H_{\boldsymbol{\xi}}^+, \; \boldsymbol{\xi}}f\!\cdot \! f\bigg)^2+\bigg(\sum\limits_{H_{\boldsymbol{\xi}}^-, \;\boldsymbol{\xi}}f\!\cdot \! f\bigg)^2\bigg\}.
    \end{split}
\end{equation*}
It is crucial to note that $2-\frac{M_q}{q-1}=\frac{5q^3+q^2-5q-2}{(q^2-1)^2}\geq 0$ for all $q>2$. In the first inequality above, equality occurs if and only if $f(\boldsymbol{\xi}-\boldsymbol{\eta})=f(\boldsymbol{\eta})C^\pm(\boldsymbol{\xi})$ for all $\boldsymbol{\eta}\in H_{\boldsymbol{\xi}}^\pm, \, \forall \boldsymbol{\xi}\in \Gamma^3$. In the second inequality, equality occurs if and only if the factors 
\begin{equation*}
    \sum\limits_{H_{\boldsymbol{\xi}}^+, \; \boldsymbol{\xi}}f\!\cdot \! f\;\; \text{ and }\;\; \sum\limits_{H_{\boldsymbol{\xi}}^-, \; \boldsymbol{\xi}}f\!\cdot \! f
\end{equation*} are equal.

Putting all the terms obtained from the previous inequalities, we get
\begin{equation}\label{eq: upper bound for main estimate}
    \begin{split}
        LHS \text{ of }\eqref{eq:mainestimate}\leq &\sum\limits_{\boldsymbol{\xi} \in \Gamma^3}\bigg( \sum\limits_{\mathcal{L}^\circ_{\boldsymbol{\xi}}, \; \boldsymbol{\xi}} f \! \cdot \! f \bigg)^2+\bigg(2-\frac{M_q}{2(q-1)}\bigg)\sum_{\boldsymbol{\xi}\in\Gamma^3}\Bigg(\bigg( \sum\limits_{H_{\boldsymbol{\xi}}^+, \; \boldsymbol{\xi}}f \! \cdot \! f\bigg)^2+\bigg( \sum\limits_{H_{\boldsymbol{\xi}}^-, \; \boldsymbol{\xi}}f \! \cdot \! f\bigg)^2 \Bigg)\\
        &+2\sum_{\boldsymbol{\xi}\in\Gamma^3}\bigg(\sum\limits_{\mathcal{L}^\circ_{\boldsymbol{\xi}}, \; \boldsymbol{\xi}}f \! \cdot \! f \bigg)\bigg(\sum\limits_{H_{\boldsymbol{\xi}}^+\cup H_{\boldsymbol{\xi}}^-, \; \boldsymbol{\xi}}f \! \cdot \! f\bigg)\\
        &-q(q+1)\sum\limits_{\Gamma^3}f^4-q(q+1)\sum\limits_{\boldsymbol{\xi} \in \Gamma^3} \bigg\{\sum\limits_{ \mathcal{L}^\circ_{\boldsymbol{\xi}}, \; \boldsymbol{\xi}}f^2 \! \cdot \! f^2+\sum\limits_{H_{\boldsymbol{\xi}}^+\cup H_{\boldsymbol{\xi}}^-, \; \boldsymbol{\xi}}f^2 \! \cdot \! f^2\bigg\}\\
        &-\frac{M_q}{q-1} \sum_{\boldsymbol{\xi}\in \Gamma^3}\bigg\{ \bigg(\sum_{  \mathcal{L}_{\boldsymbol{\xi}}} f^2\bigg)^2+ \bigg(\sum_{  \mathcal{L}_{\boldsymbol{\xi}}} f^2\bigg)\bigg(\sum_{H_{\boldsymbol{\xi}}^+\cup H_{\boldsymbol{\xi}}^-} f^2\bigg)\bigg\}.
    \end{split}
\end{equation}
Hence, it suffices to prove that the quantity on the right of the above inequality is non-positive. Now, we will divide the expression into several parts in a meaningful way and show that each part we consider is non-positive.

\emph{Step 5: Regroup terms of the same plane.} We examine each plane $A_{\boldsymbol{\xi}}^\pm$ separately. It follows from Proposition \ref{prop: properties of planes} that there are exactly $q+1$ pairwise disjoint planes of the form $A_{\boldsymbol{\xi}}^+$; the same conclusion is true for $A_{\boldsymbol{\xi}}^-$. We enumerate them as $A_i^\pm$, for $1\leq i\leq q+1$. Thus, we have the following identity
\begin{equation*}
    \sum_{\boldsymbol{\xi}\in \Gamma^3} g(\boldsymbol{\xi})=\sum_{i=1}^{q+1}\sum_{\boldsymbol{\xi} \in A_i^+}g(\boldsymbol{\xi})=\sum_{i=1}^{q+1}\sum_{\boldsymbol{\xi} \in A_i^-}g(\boldsymbol{\xi}).
\end{equation*}

Now, we pick a plane $A$ among the planes $A_i^+$, and analyze the following expression
\begin{equation*}
    \begin{split}
        S(A)\coloneqq\sum_{\boldsymbol{\xi}\in A}\frac{1}{2}&\bigg( \sum\limits_{\mathcal{L}^\circ_{\boldsymbol{\xi}}, \; \boldsymbol{\xi}}f \! \cdot \! f\bigg)^2 +\bigg(2-\frac{M_q}{2(q-1)}\bigg)\sum_{\boldsymbol{\xi}\in A}\bigg( \sum\limits_{H_{\boldsymbol{\xi}}^+, \; \boldsymbol{\xi}}f \! \cdot \! f\bigg)^2\\
        &+2\sum_{\boldsymbol{\xi}\in A}\bigg(\sum\limits_{\mathcal{L}^\circ_{\boldsymbol{\xi}}, \; \boldsymbol{\xi}}f \! \cdot \! f\bigg)\bigg(\sum\limits_{H_{\boldsymbol{\xi}}^+, \; \boldsymbol{\xi}}f \! \cdot \! f\bigg)-\frac{q(q+1)}{2}\sum_{\boldsymbol{\xi}\in A}\sum\limits_{\mathcal{L}^\circ_{\boldsymbol{\xi}}, \; \boldsymbol{\xi}}f^2  \! \cdot \! f^2 \\
        &-q(q+1)\sum_{\boldsymbol{\xi}\in A}\sum\limits_{H_{\boldsymbol{\xi}}^+, \; \boldsymbol{\xi}}f^2  \! \cdot \! f^2-\frac{q(q+1)}{2}\sum_{A}f^4\\
        &-\frac{M_q}{2(q-1)} \sum_{\boldsymbol{\xi}\in A} \bigg(\sum_{\mathcal{L}_{\boldsymbol{\xi}}} f^2\bigg)^2-\frac{M_q}{q-1} \sum_{\boldsymbol{\xi}\in A} \bigg(\sum_{\mathcal{L}_{\boldsymbol{\xi}}} f^2\bigg)\bigg(\sum_{H_{\boldsymbol{\xi}}^+} f^2\bigg).
    \end{split}
\end{equation*}
Note that some of the terms in \eqref{eq: upper bound for main estimate} appear in two different places, which are $S(A_i^+)$ and $S(A_j^-)$ for appropriate $i, j$. Hence, we are taking such terms with the coefficient $1/2$ in the definition of $S(A)$.

To estimate the quantity $S(A)$, we use the following inequalities
\begin{align*}
    \bigg( \sum\limits_{ \mathcal{L}^\circ_{\boldsymbol{\xi}}, \; \boldsymbol{\xi}}f \! \cdot \! f\bigg)^2&\leq (q-2)\sum\limits_{ \mathcal{L}^\circ_{\boldsymbol{\xi}}, \; \boldsymbol{\xi}}f ^2 \! \cdot \! f^2,\\
    \bigg( \sum\limits_{H_{\boldsymbol{\xi}}^+, \; \boldsymbol{\xi}}f \! \cdot \! f\bigg)^2&\leq q(q-1)\sum\limits_{H_{\boldsymbol{\xi}}^+, \; \boldsymbol{\xi}}f ^2 \! \cdot \! f^2,\\
    \bigg(\sum\limits_{ \mathcal{L}^\circ_{\boldsymbol{\xi}}, \; \boldsymbol{\xi}}f \! \cdot \! f\bigg)\bigg(\sum\limits_{H_{\boldsymbol{\xi}}^+, \; \boldsymbol{\xi}}f \! \cdot \! f\bigg)&\leq \bigg(\sum\limits_{ \mathcal{L}^\circ_{\boldsymbol{\xi}}}f^2\bigg)\bigg(\sum\limits_{H_{\boldsymbol{\xi}}^+}f^2\bigg),
\end{align*}
where the first and the second follow from the Cauchy-Schwarz inequality, and the last one can be deduced from the AM-GM inequality. Therefore, the equality in these inequalities occurs if and only if the following conditions are satisfied, respectively:
\begin{enumerate}[label=\textnormal{(}\it{\roman*}\textnormal{)}]
    \itemsep-0.7em 
    \item $f(\boldsymbol{\eta})=C_1(\boldsymbol{\xi})$ for all $\boldsymbol{\eta}\in \mathcal{L}^\circ_{\boldsymbol{\xi}}$;
    \item $f(\boldsymbol{\eta})=C_2(\boldsymbol{\xi})$ for all $\boldsymbol{\eta}\in H^+_{\boldsymbol{\xi}}$;
    \item $f(\boldsymbol{\eta})=f(\boldsymbol{\xi}-\boldsymbol{\eta})$ for all $\boldsymbol{\eta}\in \mathcal{L}^\circ_{\boldsymbol{\xi}}\cup H^+_{\boldsymbol{\xi}}$.
\end{enumerate}
Thus, we have the following
\begin{equation*}
    \begin{split}
        S(A)\leq\frac{q-2}{2}&\sum_{\boldsymbol{\xi}\in A} \sum\limits_{\mathcal{L}^\circ_{\boldsymbol{\xi}}, \; \boldsymbol{\xi}}f^2  \! \cdot \! f^2+\bigg(2-\frac{M_q}{2(q-1)}\bigg)q(q-1)\sum_{\boldsymbol{\xi}\in A}\sum\limits_{ H_{\boldsymbol{\xi}}^+, \; \boldsymbol{\xi}}f^2  \! \cdot \! f^2\\
        &+2\sum_{\boldsymbol{\xi}\in A}\bigg(\sum\limits_{\mathcal{L}^\circ_{\boldsymbol{\xi}}}f^2\bigg)\bigg(\sum\limits_{ H_{\boldsymbol{\xi}}^+}f^2\bigg)-\frac{q(q+1)}{2}\sum_{\boldsymbol{\xi}\in A}\sum\limits_{\mathcal{L}^\circ_{\boldsymbol{\xi}}, \; \boldsymbol{\xi}}f^2  \! \cdot \! f^2 \\
        &-q(q+1)\sum_{\boldsymbol{\xi}\in A}\sum\limits_{ H_{\boldsymbol{\xi}}^+, \; \boldsymbol{\xi}}f^2  \! \cdot \! f^2-\frac{q(q+1)}{2}\sum_{A}f^4\\
        &-\frac{M_q}{2(q-1)} \sum_{\boldsymbol{\xi}\in A} \bigg(\sum_{\mathcal{L}_{\boldsymbol{\xi}}} f^2\bigg)^2-\frac{M_q}{q-1} \sum_{\boldsymbol{\xi}\in A} \bigg(\sum_{\mathcal{L}_{\boldsymbol{\xi}}} f^2\bigg)\bigg(\sum_{ H_{\boldsymbol{\xi}}^+} f^2\bigg).
    \end{split}
\end{equation*}
At this moment, we are able to apply Fubini's theorem for the iterated sums and simplify the expression on the right. Using the facts from Proposition \ref{prop: properties of planes}, one can see that the following identities hold:
\begin{align*}
    \sum_{\boldsymbol{\xi}\in A}\sum\limits_{ \mathcal{L}^\circ_{\boldsymbol{\xi}}, \; \boldsymbol{\xi}}f^2  \! \cdot \! f^2&= \sum_{\boldsymbol{\xi} \in A} f(\boldsymbol{\xi})^2 \sum\limits_{\mathcal{L}^\circ_{\boldsymbol{\xi}}} f^2=\frac{1}{q-1}\sum_{\boldsymbol{\xi} \in A}\bigg(\sum\limits_{\mathcal{L}_{\boldsymbol{\xi}}} f^2\bigg)^2-\sum_{A} f^4,\\
    \sum_{\boldsymbol{\xi}\in A}\sum\limits_{H_{\boldsymbol{\xi}}^+, \; \boldsymbol{\xi}}f^2  \! \cdot \! f^2&=\sum_{\boldsymbol{\xi}\in A} f(\boldsymbol{\xi})^2\sum\limits_{H_{\boldsymbol{\xi}}^+}f^2 =\frac{1}{q-1} \sum_{\boldsymbol{\xi}\in A} \bigg(\sum_{\mathcal{L}_{\boldsymbol{\xi}}} f^2\bigg)\bigg(\sum_{H_{\boldsymbol{\xi}}^+} f^2\bigg),\\
    \sum_{\boldsymbol{\xi}\in A}\bigg(\sum\limits_{\mathcal{L}^\circ_{\boldsymbol{\xi}}}f^2\bigg)\bigg(\sum\limits_{H_{\boldsymbol{\xi}}^+}f^2\bigg)&=\sum_{\boldsymbol{\xi}\in A}\bigg(\sum\limits_{\mathcal{L}_{\boldsymbol{\xi}}}f^2\bigg)\bigg(\sum\limits_{H_{\boldsymbol{\xi}}^+}f^2\bigg)-\sum_{\boldsymbol{\xi}\in A}f(\boldsymbol{\xi})^2\bigg(\sum\limits_{H_{\boldsymbol{\xi}}^+}f^2\bigg)\\
    &=\frac{q-2}{q-1}\sum_{\boldsymbol{\xi}\in A}\bigg(\sum\limits_{\mathcal{L}_{\boldsymbol{\xi}}}f^2\bigg)\bigg(\sum\limits_{H_{\boldsymbol{\xi}}^+}f^2\bigg).
\end{align*}
From these identities, it follows that 
\begin{align*}
        S(A)\leq &\Bigg(\bigg(\frac{q-2}{2}-\frac{q(q+1)}{2}\bigg)\frac{1}{q-1}-\frac{M_q}{2(q-1)}\Bigg)\sum_{\boldsymbol{\xi} \in A}\bigg(\sum\limits_{\mathcal{L}_{\boldsymbol{\xi}}} f^2\bigg)^2\\
        &+\bigg(2q-\frac{q M_q}{2(q-1)}-\frac{q(q+1)}{q-1}\bigg)\sum_{\boldsymbol{\xi}\in A} \bigg(\sum_{\mathcal{L}_{\boldsymbol{\xi}}} f^2\bigg)\bigg(\sum_{H_{\boldsymbol{\xi}}^+} f^2\bigg)\\
        &+\bigg(\frac{2(q-2)}{q-1}-\frac{M_q}{q-1}\bigg)\sum_{\boldsymbol{\xi}\in A}\bigg(\sum\limits_{\mathcal{L}_{\boldsymbol{\xi}}}f^2\bigg)\bigg(\sum\limits_{H_{\boldsymbol{\xi}}^+}f^2\bigg)\\
        &+\bigg(-\frac{q-2}{2}+\frac{q(q+1)}{2}-\frac{q(q+1)}{2}\bigg)\sum_{A} f^4\\
        =&-\frac{q^5+3 q^4-4 q^3-4 q^2+3 q+2}{2 \left(q^2-1\right)^2}\sum_{\boldsymbol{\xi} \in A}\bigg(\sum\limits_{\mathcal{L}_{\boldsymbol{\xi}}} f^2\bigg)^2\\
        &+\frac{q \left(q^3+3 q^2-3 q-4\right)}{2 \left(q^2-1\right)^2}\sum_{\boldsymbol{\xi}\in A}\bigg(\sum\limits_{\mathcal{L}_{\boldsymbol{\xi}}}f^2\bigg)\bigg(\sum\limits_{H_{\boldsymbol{\xi}}^+}f^2\bigg)- \frac{q-2}{2}\sum_{A} f^4.
\end{align*}
Finally, we need one more identity to make the expression simpler
\begin{equation*}
    \begin{split}
        \sum_{\boldsymbol{\xi}\in A}\bigg(\sum\limits_{\mathcal{L}_{\boldsymbol{\xi}}}f^2\bigg)\bigg(\sum\limits_{H_{\boldsymbol{\xi}}^+}f^2\bigg)=&\sum_{\boldsymbol{\xi}\in A}\bigg(\sum\limits_{\mathcal{L}_{\boldsymbol{\xi}}}f^2\bigg)\bigg(\sum\limits_{A}f^2\bigg)-\sum_{\boldsymbol{\xi}\in A}\bigg(\sum\limits_{\mathcal{L}_{\boldsymbol{\xi}}}f^2\bigg)^2\\
        &=(q-1)\bigg(\sum\limits_{A}f^2\bigg)^2-\sum_{\boldsymbol{\xi}\in A}\bigg(\sum\limits_{\mathcal{L}_{\boldsymbol{\xi}}}f^2\bigg)^2.
    \end{split}
\end{equation*}
Hence, the above estimate simplifies to the following
\begin{equation*}
    \begin{split}
        S(A)\leq &-\frac{q^4+3 q^3-4 q^2-3 q+2}{2 (q-1)^2 (q+1)}\sum_{\boldsymbol{\xi}\in A}\bigg(\sum\limits_{\mathcal{L}_{\boldsymbol{\xi}}}f^2\bigg)^2-\frac{q-2}{2}\sum_{A} f^4\\
        &+\frac{q \left(q^3+3 q^2-3 q-4\right)}{2 (q-1)\left(q+1\right)^2}\bigg(\sum\limits_{A}f^2\bigg)^2.
    \end{split}
\end{equation*}
Note the following implication of the Cauchy-Schwarz inequality
\begin{equation*}
\sum_{A} f^4\geq \frac{1}{q^2-1}  \bigg(\sum\limits_{A}f^2\bigg)^2, 
\end{equation*}
where the equality happens if and only if $f$ is a constant function on $A$.

Thus, it follows that 
\begin{equation*}
\begin{split}
    S(A)\leq& \frac{q^4+3 q^3-4 q^2-3 q+2}{2 (q+1)^2 (q-1)}\bigg(\sum\limits_{A}f^2\bigg)^2-\frac{q^4+3 q^3-4 q^2-3 q+2}{2 (q-1)^2 (q+1)}\sum_{\boldsymbol{\xi}\in A}\bigg(\sum\limits_{\mathcal{L}_{\boldsymbol{\xi}}}f^2\bigg)^2\\
    &=\frac{q^4+3 q^3-4 q^2-3 q+2}{2 (q+1)^2 (q-1)}\Bigg(\bigg(\sum\limits_{A}f^2\bigg)^2-\frac{q+1}{q-1}\sum_{\boldsymbol{\xi}\in A}\bigg(\sum\limits_{\mathcal{L}_{\boldsymbol{\xi}}}f^2\bigg)^2\Bigg).
\end{split}
\end{equation*}
We know that the first factor is a non-negative number for all $q>2$. Moreover, notice that 
\begin{equation*}
\begin{split}
    \bigg(\sum_{A} f^2\bigg)^2=\bigg(\frac{1}{q-1}\sum_{\boldsymbol{\xi}\in A}\bigg(\sum_{\mathcal{L}_{\boldsymbol{\xi}}} f^2\bigg)\bigg)^2\stackrel{C-S}{\leq}\frac{q^2-1}{(q-1)^2}\sum_{\boldsymbol{\xi}\in A}\bigg(\sum_{\mathcal{L}_{\boldsymbol{\xi}}} f^2\bigg)^2\\
    =\frac{q+1}{q-1}\sum_{\boldsymbol{\xi}\in A}\bigg(\sum_{\mathcal{L}_{\boldsymbol{\xi}}} f^2\bigg)^2.
    \end{split}
\end{equation*}
Therefore, we conclude that $S(A)\leq 0$. The equality happens if and only if all the inequalities above are equalities, which occurs if and only if $f$ is a constant function on $A$.
To complete the proof, we note that the choice of $A$ was arbitrary; we can argue for any plane $A_i^\pm$ in a similar way. Combining all such inequalities, we deduce that \eqref{eq:mainestimate} holds. Equality in \eqref{eq:mainestimate} is attained if and only if $f$ is constant on each of the planes $A_i^\pm$. By Proposition \ref{prop: properties of planes}, we know $A_i^+$ and $A_j^-$ have a non-trivial intersection; therefore, the constants are the same. Varying $i$ and $ j$, we see that $f$ has to be a constant function on $\Gamma^3$.

\section{Classification of the extremizers: Proof of Theorem \ref{thm:classification}}\label{sec:classification}

In this section, we will give a complete classification of the functions $f$ on $\Gamma^3$ for which the inequality \eqref{eq:combinatorialestimate} becomes equality. We assume henceforth that the extremizers are not constantly zero. We begin with the following observation. 

\begin{lemma}
    Let $f\not \equiv 0$ be a function on $\Gamma^3$ for which \eqref{eq:combinatorialestimate} becomes equality. Then the following must hold:
    \begin{enumerate}[label=\textnormal{(}\it{\roman*}\textnormal{)}]
        \item $\lvert f\rvert$ is a constant function;
        \item If $\varphi(\boldsymbol{\xi})\coloneqq f(\boldsymbol{\xi})/\lvert f(\boldsymbol{\xi})\rvert$, then 
        \begin{equation}\label{eq:functionalequation}
            \varphi(\boldsymbol{x})\varphi(\boldsymbol{y})=\varphi(\boldsymbol{z})\varphi(\boldsymbol{w}),
        \end{equation}
        for all $\boldsymbol{x}, \boldsymbol{y}, \boldsymbol{z}, \boldsymbol{w} \in \Gamma^3$ such that $\boldsymbol{x}+\boldsymbol{y}=\boldsymbol{z}+\boldsymbol{w}$.
    \end{enumerate}
\end{lemma} 
\begin{proof}
    Since the estimate \eqref{eq:combinatorialestimate} becomes an equality, it must be true that the inequality \eqref{eq:multilinear form inequality} also becomes an equality, and the symmetrization $f_\sharp$ of the function $f$ has to be a constant function on $\Gamma^3$.
    Thus, we got the necessary condition on the magnitude of the extremizers. i.e. $\lvert f(\boldsymbol{\eta})\rvert^2+\lvert f(-\boldsymbol{\eta})\rvert^2=const\neq 0$ for all $\boldsymbol{\eta}\in \Gamma^3$. 
 Combining the conditions for the equality of \eqref{eq:multilinear form inequality}  and the fact that $f_\sharp=const$, we deduce that we must have the following identities, 
 \begin{equation*}\label{eq:collinearity condition}
 \begin{split}
     f(\boldsymbol{\eta}_1)&=C(\boldsymbol{\eta}_1, \boldsymbol{\eta}_2){f(-\boldsymbol{\eta}_2)};\\
     \overline{f(-\boldsymbol{\eta}_1)}&=C(\boldsymbol{\eta}_1, \boldsymbol{\eta}_2)\overline{f(\boldsymbol{\eta}_2)},
\end{split}
 \end{equation*}
 for all $\boldsymbol{\eta}_1, \boldsymbol{\eta}_2 \in\Gamma^3$ and for some number $C(\boldsymbol{\eta}_1, \boldsymbol{\eta}_2)\in\mathbb{C}$. Looking at the sum of the squares of these two equations, we deduce that $\lvert C(\boldsymbol{\eta}_1, \boldsymbol{\eta}_2)\rvert=1$, therefore, $|f|$ is constant.  Using this fact, we conclude that $f(\boldsymbol{\rho})\neq 0$ for all $\boldsymbol{\rho}\in \Gamma^3$; otherwise we see $f\equiv0$. Therefore, extremizers never vanish and we define the phase function $\varphi:\Gamma^3\to\mathbb{S}^1$ as $\varphi(\boldsymbol{\rho})\coloneqq f(\boldsymbol{\rho})/\lvert f(\boldsymbol{\rho})\rvert$.
 
 The other condition for the equality in \eqref{eq:multilinear form inequality} is the following 
 \begin{equation}\label{eq:non-negative condition}
   C(\boldsymbol{\eta}_1, \boldsymbol{\eta}_2)f_\sharp(-\boldsymbol{\eta}_2) f(\boldsymbol{\eta}_3)\overline{f(-\boldsymbol{\eta}_4)}\geq 0,
 \end{equation}
for all $\boldsymbol{\eta}_1+\boldsymbol{\eta}_2+\boldsymbol{\eta}_3+\boldsymbol{\eta}_4=0, \; \boldsymbol{\eta}_i\in \Gamma^3$. First of all, we observe that each of the factors in the expression above is non-zero by the previous deductions; therefore, the inequality must be strict. 
Hence, the condition in the equation \eqref{eq:non-negative condition} is equivalent to 
\begin{equation*}
    \operatorname{Arg}\bigg(C(\boldsymbol{\eta}_1, \boldsymbol{\eta}_2)f(\boldsymbol{\eta}_3)\overline{f(-\boldsymbol{\eta}_4)}\bigg)=0 \pmod{ 2\pi},
\end{equation*}
for all $\boldsymbol{\eta}_1+\boldsymbol{\eta}_2+\boldsymbol{\eta}_3+\boldsymbol{\eta}_4=0, \; \boldsymbol{\eta}_i\in \Gamma^3$. 
Clearly, by varying the elements $\boldsymbol{\eta}_3$ and $\boldsymbol{\eta}_4$ in this condition we conclude that 
\begin{equation*}\label{eq:equality of arguments}
    \operatorname{Arg}\bigg(f(\boldsymbol{\eta}_3)\overline{f(-\boldsymbol{\eta}_4)}\bigg)= \operatorname{Arg}\bigg(f(\boldsymbol{\eta}_3^\prime)\overline{f(-\boldsymbol{\eta}_4^\prime)}\bigg) \pmod{2\pi},
\end{equation*}
for all $\boldsymbol{\eta}_3+\boldsymbol{\eta}_4=\boldsymbol{\eta}_3^\prime+\boldsymbol{\eta}_4^\prime\neq 0$. 
If we write this equation in terms of the phase function, we get
\begin{equation*}
    \frac{\varphi(\boldsymbol{\eta}_3)}{\varphi(-\boldsymbol{\eta}_4)}= \frac{\varphi(\boldsymbol{\eta}^\prime_3)}{\varphi(-\boldsymbol{\eta}^\prime_4)}.
\end{equation*}
After we shuffle the terms and rename the variables, we arrive at the functional equation \eqref{eq:functionalequation} for $\varphi$.
\end{proof}


In the rest of this section, we study the solutions to the functional equation \eqref{eq:functionalequation}.

\begin{proposition}\label{prop:classificationonplanes}
    Let $p$ be an odd prime and $q=p^n$. If a map $\psi:\mathbb{F}_q^2\setminus\{\boldsymbol{0}\}\to \mathbb{C}^\times$ satisfies
    \begin{equation}\label{eq:functionalequationonplane}
        \psi(\boldsymbol{x})\psi(\boldsymbol{y})=\psi(\boldsymbol{z})\psi(\boldsymbol{w})
    \end{equation}
    for all non-zero $\boldsymbol{x}, \boldsymbol{y}, \boldsymbol{z}$, and $\boldsymbol{w}$ such that $\boldsymbol{x}+\boldsymbol{y}=\boldsymbol{z}+\boldsymbol{w}$, then it must be a multiple of a character. More precisely,  there exists $\lambda\in\mathbb{C}^\times$ and $(a_1, a_2)\in \mathbb{F}_q^2$ such that 
    \begin{equation}\label{eq:solutionontheplane}
        \psi(x_1, x_2)=\lambda \cdot\exp{\left(\frac{2 \pi i}{p}\operatorname{Tr}_n(a_1x_1+a_2x_2)\right)},
    \end{equation}
    for all $(x_1, x_2)\in \mathbb{F}_q^2\setminus\{\boldsymbol{0}\}$.
\end{proposition}
\begin{proof}
    From \eqref{eq:functionalequationonplane} it follows that 
    $\psi(\boldsymbol{x})\psi(-\boldsymbol{x})=const$, for all $\boldsymbol{x}\neq\boldsymbol{0}$. Since every non-zero scalar multiple of $\psi$ satisfies the functional equation \eqref{eq:functionalequationonplane}, we may assume (up to a constant factor) that 
    \begin{equation}\label{eq:normalizationcondition}
        \psi(\boldsymbol{x})\psi(-\boldsymbol{x})=1, \quad \text{ for all } \boldsymbol{x}\neq\boldsymbol{0}.
    \end{equation} Under these conditions, we still have two choices for the sign of $\psi$; thus, we can try to identify the solutions up to a factor of $\pm 1$.
    
    \emph{Step 1:} $\psi(\boldsymbol{x})^{2p}=1$, for all $\boldsymbol{x}\neq\boldsymbol{0}$.

Fix $\boldsymbol{x}\neq\boldsymbol{0}$. Since there are at least $3$ elements in $\mathbb{F}_q^2\setminus\{\boldsymbol{0}\}$ one finds $\boldsymbol{y}\in \mathbb{F}_q^2\setminus\{\boldsymbol{0}\}$ such that $\boldsymbol{y}\neq\pm \boldsymbol{x}$. Therefore, using  \eqref{eq:normalizationcondition} and \eqref{eq:functionalequationonplane} consecutively, we deduce
\begin{equation}\label{eq:quadraticrelation}
    \psi(\boldsymbol{x})^2=\psi(\boldsymbol{x})\psi(\boldsymbol{y})\psi(-\boldsymbol{y})\psi(\boldsymbol{x})=\psi\left(\frac{\boldsymbol{x}+\boldsymbol{y}}{2}\right)^2\psi\left(\frac{\boldsymbol{x}-\boldsymbol{y}}{2}\right)^2=\psi\left(\frac{\boldsymbol{x}}{2}\right)^4.
\end{equation}
If we assume $p\neq3$, again using \eqref{eq:normalizationcondition} we can obtain the relation 
\begin{equation*}
    \psi(3\boldsymbol{x})=\psi(3\boldsymbol{x})\psi(-\boldsymbol{x})\psi(\boldsymbol{x})=\psi(\boldsymbol{x})^3.
\end{equation*}
Repeating similar argument several times we show that $\psi((p-2)\boldsymbol{x})=\psi(\boldsymbol{x})^{p-2}$, which is also equivalent to the identity $\psi(2\boldsymbol{x})=\psi(\boldsymbol{x})^{2-p}$. Combining the last identity with \eqref{eq:quadraticrelation} we get 
\begin{equation*}
\psi(\boldsymbol{x})^4=\psi(2\boldsymbol{x})^2=\psi(\boldsymbol{x})^{4-2p},
\end{equation*}
which implies that $\psi(\boldsymbol{x}
)^{2p}=1$ for the case $p\neq3$.\\
Now suppose $p=3$. In this case, we cannot use the previous argument simply because $3x=0$. What we can use is the fact that $2^{-1}=2$ and \eqref{eq:quadraticrelation}. Employing these facts, one can show
\begin{equation*}
\psi(\boldsymbol{x})^2=\psi(2\boldsymbol{x})^4=\psi(4\boldsymbol{x})^8=\psi(\boldsymbol{x})^8.
\end{equation*}
Therefore, $\psi(\boldsymbol{x})^6=1$. It completes the proof of step 1.

\emph{Step 2:} $\psi(2\boldsymbol{x})=\pm\psi(\boldsymbol{x})^2$, where the choice of $\pm$ is independent of $\boldsymbol{x}$.

From \eqref{eq:quadraticrelation} it is immediate that $\psi(2\boldsymbol{x})\in\{\pm\psi(\boldsymbol{x})^2\}$. We argue by contradiction, suppose there are two elements $\boldsymbol{x}, \boldsymbol{y}$ such that $\psi(2\boldsymbol{x})=\psi(\boldsymbol{x})^2$ and $\psi(2\boldsymbol{y)}=-\psi(\boldsymbol{y})^2$. In such a case, using the conclusion of step 1, we are led to the conclusion
\begin{equation*}
    -1=(-1)^p=\left(-\psi(\boldsymbol{x})^2\psi(\boldsymbol{y})^2\right)^p=\psi(2\boldsymbol{x})^p\psi(2\boldsymbol{y}
    )^p=\psi(\boldsymbol{x}+\boldsymbol{y})^{2p}=1,
\end{equation*}
which is absurd. Thus, the choice of the sign in the identity is uniform.

Using step 2, we can assume that $\psi(2\boldsymbol{x})=\psi(\boldsymbol{x})^2$; otherwise, we work with $-\psi$, which is also a solution to the equation \eqref{eq:functionalequationonplane} with \eqref{eq:normalizationcondition}.
Define an extension $\Psi$ of the function $\psi$ given by 
\begin{equation*}
    \Psi(\boldsymbol{x})\coloneqq\left\{\begin{aligned}
&\psi(\boldsymbol{x}),& &\boldsymbol{x}\in \mathbb{F}_q^2\setminus\{\boldsymbol{0}\};\\
&1,& &\boldsymbol{x}=\boldsymbol{0}.
\end{aligned}\right.
\end{equation*}

From this definition, it is immediate that $\Psi(\boldsymbol{x})\Psi(\boldsymbol{y})=\Psi(\boldsymbol{z})\Psi(\boldsymbol{w})$ holds for all $\boldsymbol{x}+\boldsymbol{y}=\boldsymbol{z}+\boldsymbol{w}$. In particular, $\boldsymbol{x}+\boldsymbol{y}=(\boldsymbol{x}+\boldsymbol{y})+\boldsymbol{0}$, implies 
\begin{equation*}
    \Psi(\boldsymbol{x})\Psi(\boldsymbol{y})=\Psi(\boldsymbol{x}+\boldsymbol{y}),
\end{equation*}
for all $\boldsymbol{x}, \boldsymbol{y}$. Thus $\Psi$ is a character on the additive group $(\mathbb{F}_q^2, +)$. Now it is clear that there exists $(a_1, a_2)\in \mathbb{F}_q^2$ such that 
\begin{equation*}
    \Psi(\boldsymbol{x})=\exp{\left(\frac{2 \pi i}{p}\operatorname{Tr}_n(a_1x_1+a_2x_2)\right)},
\end{equation*}
for all $\boldsymbol{x}=(x_1, x_2)\in \mathbb{F}_q^2$. In particular, $\psi(\boldsymbol{x})$ has the same representation for all $\boldsymbol{x}\neq \boldsymbol{0}$. Since we have done several reductions on the scalar factor, the general solution to the functional equation \eqref{eq:functionalequation} is given by \eqref{eq:solutionontheplane}.
\end{proof}

Now, using this result, we can classify all the extremizers.  We know that they are constant in absolute value; therefore, we can assume $f(\boldsymbol{x})=\lambda \: \varphi(\boldsymbol{x})$, for some $\lambda\in \mathbb{C}^\times$, and here $\varphi$ is the phase function of $f$. We first observe that any map $ f: \Gamma^3 \to \mathbb{C}$ which acts as follows
\begin{equation*}
     (\eta_1,\eta_2,\eta_3,\eta_4) \mapsto \lambda \cdot \exp\left(\frac{2\pi i}{p} \operatorname{Tr}_{n}(a_1 \eta_1 + a_2 \eta_2 + a_3 \eta_3 + a_4 \eta_4)\right) 
\end{equation*}
for given $ a_1, a_2, a_3, a_4 \in \mathbb{F}_q $ and $ \lambda \in \mathbb{C}^\times $ is an extremizer for the inequality \eqref{eq:combinatorialestimate}. Indeed, 
\begin{equation*}
\begin{split}
    \sum_{\boldsymbol{\xi}\in \Fq^4}\bigg\lvert\sum_{\substack{\boldsymbol{\eta}_1, \boldsymbol{\eta}_2\in \Gamma^3\\ \boldsymbol{\eta}_1+\boldsymbol{\eta}_2=\boldsymbol{\xi}}} f(\boldsymbol{\eta}_1)f(\boldsymbol{\eta}_2) \bigg\rvert^2&=\lvert \lambda\rvert^4 \sum_{\boldsymbol{\xi}\in \Fq^4}\bigg\lvert\sum_{\substack{\boldsymbol{\eta}_1, \boldsymbol{\eta}_2\in \Gamma^3\\ \boldsymbol{\eta}_1+\boldsymbol{\eta}_2=\boldsymbol{\xi}}} \exp \left(\frac{2 \pi i}{p}\operatorname{Tr}_n\bigg((\boldsymbol{\eta}_1+\boldsymbol{\eta}_2)\cdot \boldsymbol{a}\bigg)\right) \bigg\rvert^2\\
    &=\lvert \lambda\rvert^4 \sum_{\boldsymbol{\xi}\in \Fq^4} \left\lvert \Sigma_{\boldsymbol{\xi}} \right\rvert^2= \mathbf{C}^\ast_{\Gamma^3}(2\to4)  \bigg(\sum_{\boldsymbol{\xi}\in \Fq^4} \lvert f(\boldsymbol{\xi})\rvert^2 \bigg)^2.
\end{split}
\end{equation*}

Now, we prove that these are the only extremizers with $|f|=\lvert \lambda\rvert$. Since the phase function $\varphi$ has modulus 1 by \eqref{eq:functionalequation} and the previous proposition, we know that $\varphi|_{A}=c \:\chi$, where $c\in \mathbb{S}^1$ and $\chi$ is a character on the complete plane $A\cup \{\boldsymbol{0}\}$, for any punctured plane $A\subset \Gamma^{3}$. Up to a constant factor, we may assume $c=1$. Hence, we have that $\varphi$ acts as a character in any punctured plane contained in the cone. In particular, for $A_1\coloneqq\left \{(0,\eta_2,0,\eta_4): (\eta_2,\eta_4)\in \mathbb{F}_q^2\setminus\{\boldsymbol{0}\}\right\}$ and $A_2\coloneqq\left\{(\eta_1,0,\eta_3,0):(\eta_1,\eta_3)\in \mathbb{F}_q^2\setminus\{\boldsymbol{0}\}\right\}$, we have: $\varphi(0,\eta_2,0,\eta_4)=\exp\left(\frac{2\pi i}{p}\operatorname{Tr}_{n}(a_2\eta_2+a_4\eta_4)\right)$ and $\varphi(\eta_1,0,\eta_3,0)=\exp\left(\frac{2\pi i}{p}\operatorname{Tr}_{n}(a_1\eta_1+a_3\eta_3)\right)$, for some $a_1,a_2,a_3,a_4\in \mathbb{F}_q$. Consider the function $\psi:\Gamma^3\to \mathbb{S}^1$ given by
\begin{equation*}
    \psi(\eta_1, \eta_2, \eta_3, \eta_4)=\frac{1}{\varphi(\eta_1, \eta_2, \eta_3, \eta_4)}\exp\left(\frac{2\pi i}{p}\operatorname{Tr}_{n}(a_1\eta_1+a_2\eta_2+a_3\eta_3+a_4\eta_4)\right).
\end{equation*} 
Note that also $\psi$ must satisfy \eqref{eq:functionalequation}. Observe that $\psi(0,\eta_2,0,\eta_4)=1$ and $\psi(\eta_1,0,\eta_3,0)=1$ for any $(\eta_1,\eta_3), (\eta_2,\eta_4) \in \mathbb{F}_q^2\setminus\{\boldsymbol{0}\}$. Now, take any $(\eta_1,\eta_2,\eta_3,\eta_4)\in \Gamma^3\setminus \left(A_1\cup A_2\right)$ and take a nonzero $(\rho_2,\rho_4)\neq (-\eta_2, -\eta_4)$; then we have that 
\begin{align*}
\psi(\eta_1,\eta_2,\eta_3,\eta_4)=&\psi(\eta_1,\eta_2,\eta_3,\eta_4)\psi(0,\rho_2,0,\rho_4)\\
&=\psi(\eta_1,0,\eta_3,0)\psi(0,\eta_2+\rho_2,0,\eta_4+\rho_4)=1.
\end{align*}
From this, we conclude that $\varphi(\boldsymbol{\eta})$ is exactly given by $\exp\left(\frac{2\pi i}{p}\operatorname{Tr}_{n}(a_1\eta_1+a_2\eta_2+a_3\eta_3+a_4\eta_4)\right)$.


\newpage
\bibliographystyle{abbrv}
\bibliography{references.bib}

\end{document}